\documentclass[leqno]{prims}
\usepackage{amsmath}
\def\qedhere{}

\usepackage{hyperref,dettweiler-sabbah}

\theoremstyle{plain}
\newtheorem{thm}{Theorem.}[section]
\newtheorem{cor}[thm]{Corollary.}
\newtheorem{lemma}[thm]{Lemma.}
\newtheorem{prop}[thm]{Proposition.}

\newtheorem{defn}[thm]{Definition.}

\newtheorem{ass}[thm]{Assumption.}
\newtheorem{rem}[thm]{Remark.}

\newcommand{\co}{{\omega}}
\newcommand{\ten}{{\otimes}}

\newcommand{\pr}{{\rm pr}}

\newcommand{\x}{{\bf{x} }}
\newcommand{\y}{{\bf{y} }}
\newcommand{\tstar}{{\tilde{\star} }}
\DeclareMathOperator{\Gr}{Gr}
\newcommand{\onto}{\to\hspace*{-5mm}\to}

\renewcommand{\PP}{\mathbb{P}}
\renewcommand{\ZZ}{\mathbb{Z}}

\newcommand{\V}{\mathcal{V}}

\renewcommand{\x}{{\bf x}}

\renewcommand{\MC}{{\rm MC}}

\renewcommand{\To}{\;\longrightarrow\;}

\newcommand{\RH}{\mathcal{RH}}

{

\begin{document}

\renewcommand{\theenumi}{\roman{enumi}}%

\TitleHead{On the Hodge theory of the additive middle convolution}
\title{On the Hodge theory of the additive  middle convolution}

\AuthorHead{M.~Dettweiler and S.~Reiter}
\author{Michael \textsc{Dettweiler}\footnote{M.~Dettweiler: Department of Mathematics,
University of Bayreuth,
95440 Bayreuth,
Germany;
\email{michael.dettweiler@uni-bayreuth.de}} and
Stefan \textsc{Reiter}\footnote{S.~Reiter: 
Department of Mathematics,
University of Bayreuth,
95440 Bayreuth,
Germany;
\email{stefan.reiter@uni-bayreuth.de}} 
}

\maketitle
\begin{abstract}
We compute the behaviour of Hodge data under  additive middle convolution for  irreducible variations of polarized complex Hodge structures on  punctured complex affine lines.
\end{abstract}
%

\section*{Introduction}

In previous work of Sabbah with one of the authors \cite{DS}, the effect of the additive middle convolution $\MC_\chi(V)=V\star L_\chi$
of a complex polarized Hodge module  $V$ on $\AA^1$ with a Kummer module $ L_\chi$ 
on  various local and global Hodge data was 
determined. This leads to an analog of Katz' algorithm for irreducible rigid local systems \cite{Katz96} in the context of Hodge modules.

It is the aim of this work to extend these results to the case of the middle convolution $V\star L$ 
(cf.~Section~\ref{secprel}) of 
two irreducible and non-constant complex polarized Hodge modules on $\AA^1.$ It turns out that,  to a large extent, 
the general case can be reduced to the middle convolution with  Kummer modules treated in \cite{DS}. \\

In Section~\ref{sectiondegrees}, Theorem~\ref{H-tensor},  the global Hodge numbers of tensor products $V\otimes L$
 (the degrees of the associated Hodge bundles)
are determined, generalizing \cite{DS}, Prop.~2.3.2. We are indepted to Claude Sabbah for communicating the proof of Theorem~\ref{H-tensor} to us.
This result is important in many applications where convolution is applied 
iteratively in combination with tensor operations (cf.~\cite{Katz96}, \cite{DS}, \cite{DR14}).  \\

In Section~\ref{seclocdata} we determine the local Hodge data  of the vanishing and 
nearby cycles (cf.~Section~\ref{secprel} and \cite{DS} for these notions) 
at the finite singularities of a convolution $V\star L$ (Theorem~\ref{thmthomseb}).  As in \cite{DS}, the main tool for doing this  
 is Saito's version of the Thom-Sebastiani theorem (cf.~\cite{DS}, Theorem~3.2.3, and  its corrigendum~\cite{DSCorr}, where a proof of
 the Thom-Sebastiani result is provided).\\

In Section~\ref{secglobal}  the global Hodge numbers of $V\star L$  are determined. 
The main observation is that the middle convolution $V\star L_\chi$ of an irreducible and nontrivial  Hodge module $V$ 
with a generic Kummer module $L_\chi$ is {\it parabolically rigid}, meaning that the associated parabolic cohomology group
$$H^1_{\rm par}(V\star L_\chi)=H^1(\PP^1, j_*\mathcal{H}^{-1}(\RH(V\star L_\chi)))$$ vanishes (where $\RH(V\star L_\chi)$ is the perverse sheaf associated to 
$V\star L_\chi$ via Riemann-Hilbert correspondence and $j$ is the projective embedding  of $\AA^1$). 
Using the Riemann-Roch theorem, a formula for the Hodge numbers of $H^1_{\rm par}(V\star L_\chi)$ involving 
local and global data was given in \cite{DS}, Proposition~2.3.3. 
Hence the vanishing of $H^1_{\rm par}(V\star L_\chi)$
gives a method to compute the global Hodge numbers of $V\star L.$\\

The remaining local Hodge data at $\infty$ of $V\star L$ are determined in Section~\ref{seclocalinfty}.  For this, we make use of hypergeometric Hodge modules with prescribed local behaviour at $\infty$ and reduce the general case to the 
convolution of these. We believe that a more conceptual proof of these results may be given in the context of 
irregular Hodge filtrations on twistor modules and their behaviour under Fourier-Laplace transformation (cf.~\cite{ESJ}). \\

In a forthcoming work, the authors prove similar results for the multiplicative convolution (also called Hadamard product). \\

\section{Preliminary results}\label{secprel}

Following \cite{DS}, we review the basic notions of middle convolutions introduced by Katz \cite{Katz96}, in the frame of holonomic $\cD$-modules.
Let
$s:\AA^1\times \AA^1 \to \AA^1$ be the addition map and  let $M,N$ be holonomic $\cD(\AA^1)$-modules. The {\it additive  $*$-convolution} $M\star_*N$ of $M$ and $N$ is the object $s_+(M\boxtimes N)$ of $\catD^\rb(\AA^1)$. 
The {\it additive    $!$-convolution} can be defined as $M\star_!N=\bD(\bD M\star_*\bD N)$, where $\bD$ is the duality functor $\catD^{\rb,\op}_{\hol}(\cD(\AA^1))\to \catD^\rb_{\hol}(\cD(\AA^1))$. It can also expressed as $s_\dag(M\boxtimes N)$, if~$s_\dag:=\bD s_+\bD$ denotes the adjoint by duality of $s_+,$ cf.~\cite{DS} (under the Riemann Hilbert correspondence, the functor $+$ corresponds to 
the derived $*$-functor  and $\dag$ corresponds to $!,$ explaining the notion).

Let us choose a projectivization $\wt s:X\to \AA^1$ of $s$, and let $j: \AA^1 \times \AA^1 \hto\nobreak X$ denote the open inclusion. Since $\wt s$ naturally commutes with duality, we have \hbox{$\wt s_\dag=\wt s_+$} and $s_\dag=\wt s_+\circ j_\dag$. Since there is a natural morphism $j_\dag\to j_+$ in $\catD^\rb_\hol(\cD_X)$, we get a functorial morphism $s_\dag(M\boxtimes N)\to s_+(M\boxtimes N)$, that is, $M\star_!N\to M\star_*N$, in $\catD^\rb_\hol(\cD(\AA^1))$.
Let  $\catP$ be  the  full subcategory  of $\Mod_\hol(\AA^1)$ consisting of
holonomic $\cD(\AA^1)$-modules  $N$ 
such that for all holonomic $\cD(\AA^1)$-modules  $M$ 
both types of convolutions $N\star_*M$ and $N\star_! M$ are again holonomic. 

\begin{definition}\label{defconv}
For $N$ in $\catP$ and $M$ holonomic, the {\it middle convolution} $M\star_\Mid N$ is defined as the image of $M\star_!N\to M\star_*N$ in $\Mod_\hol(\cD(\AA^1))$. For simplicity we often set
$M\star N:=M\star_\Mid N.$ As explained in \cite{DS}, Section~3.3, this notion extends to 
the category of complex polarized Hodge modules on $\AA^1,$ using the  results of 
\cite{MSaito86} and \cite{MSaito87}. If $M$ is smooth on $\AA^1\setminus \x\, (\x=\{x_1,\ldots,x_r\} \cup \{\infty\} )$ and if $N$
is smooth on $\AA^1\setminus \y\, (\y=\{y_1,\ldots,y_s\}\cup \{\infty\})$ then $M\star N$ as well as 
the other types of convolutions are smooth on $\AA^1\setminus \x\star \y,$ where 
$$ \x\star \y =\{x_i+y_j\mid i=1,\ldots,r,\, j=1,\ldots ,s\} \cup \{\infty\}.$$ 
\end{definition}

The following result follows from the Riemann-Hilbert correspondence and \cite{Katz96}, Cor.~2.6.10 and Cor.~2.6.17:

\begin{lemma}\label{lem:starexact} \begin{enumerate}
\item If $N$ is irreducible  
such that its isomorphism class is not translation invariant then $N$ has the property~$\catP$.
\item If $N$ and $M$ are in  $\catP$ then $N\star M$ again is in $\catP$.
\end{enumerate}
\end{lemma}

Let $W$ be a  complex polarized Hodge module on the complex affine line $\AA^1$ which is smooth on 
 $\AA^1\setminus \{z_1,\ldots,z_k\}.$ The local system on $\AA^1\setminus \{z_1,\ldots,z_k\}$ 
 which is underlying $W$ is denoted as $\mathcal{W}.$ The perverse sheaf on $\AA^1$ 
 associated to $W$ via the de Rham functor is denoted by $\RH(W)$ 
  and we view the {\it $i$-th parabolic cohomology group}
 $$ H^i_{\rm par}(W):=H^i(\PP^1,j_*\mathcal{H}^{-1}(\RH(W)))\quad (j: \AA^1 \hookrightarrow  \PP^1) $$ 
 to be equipped with its natural Hodge structure.\\

Throughout the article we will work with Hodge modules $V,L,\delta_x,L_\chi$ which are as follows: 

 \begin{ass}\label{ass1} \begin{enumerate} 
 \item We assume that $V=(V,F^\cbbullet V)$ is a
  complex polarized Hodge module on the complex affine line $\AA^1$ which is 
 the intermediate (minimal)  extension of 
 an irreducible nonconstant  variation of polarized complex Hodge structures 
  on $\AA^1\setminus\x$  
 (where $\x=\{x_1,\ldots,x_r,\infty\}\subset \PP^1(\CC)).$    In this situation we sometimes set $x_{r+1}=\infty.$  
 
 
\item  Let  $L=(L,F^\cbbullet L)$ be  another Hodge module of the same kind 
 which is the minimal extension of 
 a   variation of polarized complex Hodge structures 
  on $\AA^1\setminus \y$ 
 (where $\y=\{y_1,\ldots,y_s,\infty\}).$ 
 
\item  For a point $x\in \AA^1$ we write $\delta_x$ for the Hodge module which corresponds to the rank-one skyscraper sheaf on $\AA^1$ supported in $x,$ having     trivial Hodge filtration (so that $h^0(H^0(\AA^1,\delta_x))=1$).

\item As in \cite{DS}, Section~3.3, we  write $L_\chi$ for the Hodge module with trivial 
Hodge filtration belonging 
to the Kummer sheaf with residues $(\mu,1-\mu)\, (\mu \in (0,1))$ such that $\chi=e^{-2\pi i \mu},$ having singular points 
  at $(0,\infty).$ 
  We call $L_\chi$ {\rm generic}  if the monodromy eigenvalues 
 of all sheaves different from $L_\chi$ and $L_{\chi^{-1}}$ involved in 
our arguments  are different from $\chi^{\pm 1}.$ 
\end{enumerate}
\end{ass}

For the following notions and stated results we refer to \cite{DS}, Section~1.2 and Sections~2.2, 2.3: on the one hand, one has   {\it global Hodge data} $\delta^p(V)$
given by the degrees of the Hodge bundles.  
On the other hand, one has {\it local Hodge data}:  For each point $x\in \{x_1,\ldots,x_{r+1}\}$ and each $\lambda\in S^1$ one has the notion $\psi_{x,\lambda}(V)$
of the 
generalized $\lambda$-eigenspace of the 
nearby cycles $\psi_{x}(V).$   We will also use the corresponding notion of 
vanishing cycles $\varphi_{x,\lambda}(V).$  These spaces are mixed Hodge structures  with associated nilpotent monodromy operator, derived from the local monodromy,
which imposes an associated weight filtration $W.$ One has the notion of $l$-primitive vectors
${\rm P}_l \varphi_{x,\lambda}(V)$ with respect to the Lefschetz decomposition of 
$\varphi_{x,\lambda}(V).$ 
For $l\in \NN$ we define 
$l$-primitive local Hodge numbers as follows: 
$$ \nu_{x,\lambda,l}^p(V)=\nu_{x,a,l}^p(V):=\dim \gr^p_F {\rm P}_l\psi_{x,\lambda}(V),$$ 
where $a\in \RR\cap [0,1)$ 
such that  $\lambda=\exp(2 \pi i (-a)).$ We set 
$$ \nu_{x,a}^p(V):=\sum_{l\geq 0}\sum_{k=0}^l\nu_{x,a,l}^{p+k}(V)\quad \textrm{ and }\quad 
\nu_{x,a,{\rm prim}}^p:=\sum_{l\geq 0}\nu_{x,a,l}^p(V) $$
as well as
$$ h^p(V):=\nu_{x}^p(V):=\sum_{a\in [0,1)}\nu_{x,a}^p(V)\quad \textrm{ and }\quad 
 \nu_{x,\neq 0}^p(V):=\sum_{a\in (0,1)}\nu_{x,a}^p(V).$$

One has corresponding notions for vanishing cycles 
$$ \mu_{x,\lambda,l}^p(V)=\mu_{x,a,l}^p(V):=\dim \gr^p_F {\rm P}_l\varphi_{x,\lambda}(V),$$ 
and
$$ \mu_{x,a}^p(V):=\sum_{l\geq 0}\sum_{k=0}^l\mu_{x,a,l}^{p+k}(V)\quad \textrm{ and }\quad 
\mu_{x,a,{\rm prim}}^p:=\sum_{l\geq 0}\mu_{x,a,l}^p(V). $$ These notions are related as follows
(cf.~loc.cit.):
$$ \mu_{x,a,l}^p(V)=\nu_{x,a,l}^p(V) \textrm{ if } a\neq 0\quad \textrm{ and }\quad 
\mu_{x,0,l}^p(V)=\nu_{x,0,l+1}^p(V).$$ 

Additionally to \cite{DS}, we will use the following further local Hodge numbers, simplifying the computations
below: 

\begin{defn} \label{def1} 
Let
$$
  \co^p_{x}(V) :=\nu^p_{x}(V)-\nu^p_{x,0,\prim}(V) =\nu_{x,\neq 0}^{p}(V)+\mu_{x,0}^{p+1}(V)
  ,$$   cf.~\cite{DS},~(2.2.5*), and 
   \begin{eqnarray*}  
      \co^p_{ss,x}(V)&:=&\nu^p_{x,\neq 0}(V)\\
      \co^p_{u,x}(V)&:=&\mu^{p+1}_{x,0}(V) \\
   \co^p_{\neq \infty }(V) &:=&\sum_{x\in (\x\setminus \infty)}\co^p_{x}(V) \\
   \co_{\neq \infty}(V)&:=&\sum_p \co^p_{\neq \infty}(V) \\
    \co^p_{}(V) &:=&\sum_{x\in \x } \co^p_{x}(V) \\
     \co(V)&:=&\sum_p \co^p(V)\\ 
      \kappa^p_{x}(V) &:=& \nu^{p}_{x,0,\prim}(V).    
 \end{eqnarray*}
%
 
Let $J^p(a,l)(V)$ denote a mixed $\CC$-Hodge structure 
which is associated to  a nilpotent orbit belonging to a monodromy operator whose Jordan form is a 
single Jordan block of size $l$ and having residue $a\in [0,1)$ such that 
$\nu^p_{a,l-1}(V)=1.$ 
\end{defn}

\begin{rem}\label{rempsi1} 
One has $$  \psi_{x_j}(V) \simeq \bigoplus_{(i,a,l)}   J^i(a,l)^{\nu^i_{x_j,a,l-1}(V)}$$
(note that we use complex coefficients, so any pure Hodge structure decomposes into one-dimensional summands). 
\end{rem}

In the following, let $j: \AA^1\setminus \x \hookrightarrow \PP^1$ be the natural inclusion. 
Using our above notion of $ \co^i_{}(V)$ we obtain:

\begin{prop}\label{H-para}
 \[ h^p(H^1_{\rm par} (V))=\delta^{p-1}(V)-\delta^p(V)-h^p(V)-h^{p-1}(V)+\co^{p-1}(V). \]\end{prop}

\begin{proof}
By \cite{DS}, Proposition 2.3.3, we have 
 \begin{eqnarray*} h^p(H^1_{ \rm par }(V))&=&\delta^{p-1}(V)-\delta^p(V)-h^p(V)-\nu_{\infty,0,{\rm prim}}^{p-1}(V)+\sum_{j=1}^r(\nu_{x_j,\neq 0}^{p-1}(V)
 +\mu_{x_j,0}^{p }(V)) \\
 &=&\delta^{p-1}(V)-\delta^p(V)-h^p(V)-\nu_{\infty,0,{\rm prim}}^{p-1}(V)+\sum_{j=1}^r \co^{p-1}_{x_j}(V) \\
 &=& \delta^{p-1}(V)-\delta^p(V)-h^p(V)-h^{p-1}(V)+\co^{p-1}(V),
 \end{eqnarray*}
 using $$ \nu_{\infty,0,{\rm prim}}^{p-1}(V)=\nu_\infty^{p-1}(V)-\co_\infty^{p-1}(V)\quad \textrm{ and }
 \quad h^{p-1}(V)=\nu_\infty^{p-1}(V).$$
\end{proof}

\begin{rem}\label{remdelta} The construction of nearby and vanishing cycles and their basic invariants is 
carried out for minimal extensions in \cite{DS}, Section~2.2.  
The general case can be reduced with this at hand to the case of mixed Hodge-modules with punctual support.
The Hodge invariants of these are as follows:  Let $V$ be a Hodge module supported on a closed point 
$x\stackrel{i}{\hookrightarrow} \AA^1$ (i.e., a complex polarized Hodge structure $V$ placed at $x$) and let
$i_+V$ its extension to $\AA^1.$  By the usual triangle which connects nearby and vanishing cycles, 
 the nearby cycles of $i_+V$ are 
zero, while the vanishing cycles 
$\varphi_{x}(i_+V)$  can be identified with $V.$
 Note that the natural monodromy operation on $\varphi_{x}(i_+V)$ is trivial, 
hence $ \mu^p_{x,a}(i_+V)=0$ for $a\neq 0$ 
and 
$ \mu^p_{x,0}(i_+V)=h^p(V).$
\end{rem}

\section{Degrees of tensor products}\label{sectiondegrees}

We will proceed using the notions of the previous section.
 The following theorem is a  generalization of \cite{DS}, Proposition~2.3.2. We are indepted to 
 Claude Sabbah for communicating its proof to us.
  
\begin{thm}\label{H-tensor}
 \begin{eqnarray*} 
  \delta^l(V \ten L)&= & \sum_{p} \delta^{l-p}(V) h^p(L) +  \sum_{p} h^{l-p}(V) \delta^p(L) 
   +\sum_{x \in \x \cap \y }o^l_{x}(V \ten L), 
 \end{eqnarray*}
 where
\begin{eqnarray*} 
  o^l_{x}(V \ten L) &:=&  \sum_{p} \sum_{a+b\geq 1} \nu^p_{x,a}(V)\nu^{l-p}_{x,b}(L).
 \end{eqnarray*}

\end{thm}

The result depends on the following two lemmata.
Let $V^0,L^0,(V\otimes L)^0$ denote the Deligne extensions of $V,L,V\otimes L$ (resp.). 
There is also $(V\otimes L)^0$. We have the following Hodge filtrations:
\begin{itemize}
\item
The tensor product filtration $F^\ell(V^0\otimes L^0):=\sum_pF^{\ell-p}V^0\otimes F^pL^0$.
\item
Since $V\otimes L$ is a variation of Hodge structures  on the punctured $\PP^1$, with Hodge filtration equal to the tensor product filtration, we obtain the filtration $F^\ell(V\otimes L)^0$.
\end{itemize}

Let $D=\bmx\cup\bmy$ denote the reduced divisor away from which $V$ and $L$ are variations of Hodge structures. A local computation (without using Hodge theory) shows that there are $F$-filtered inclusions
\[
(V\otimes L)^0(-D)\subset V^0\otimes L^0\subset(V\otimes L)^0
\]
which are equalities away from $D$.

\begin{lemma}\label{lem:strictincl}
The inclusion $V^0\otimes L^0\subset(V\otimes L)^0$ is strict with respect to $F^\cbbullet$.
\end{lemma}

If this lemma is proved, we find that, for each $\ell$, there is an injective morphism
\[
\bigoplus_p\gr^{\ell-p}_FV^0\otimes\gr^pL^0\hto\gr^\ell_F(V\otimes L)^0
\]
whose cokernel is supported on $D$ and has dimension $\dim\gr^\ell_F\bigl((V\otimes L)^0/V^0\otimes L^0\bigr)$. As a consequence, we find
\[
\delta^\ell(V\otimes L)=\sum_p(\delta^{\ell-p}V\cdot h^pL+h^{p-\ell}V\cdot\delta^pL)+\dim \gr^\ell_F\Bigl(\frac{(V\otimes L)^0}{V^0\otimes L^0}\Bigr).
\]

\begin{lemma}\label{lem:strictincl2}
We have
\[
\dim\gr^\ell_F\Bigl(\frac{(V\otimes L)^0}{V^0\otimes L^0}\Bigr)=\sum_{x\in D}o_x^\ell(V\otimes L).
\]
\end{lemma}

Note that $o_x^\ell(V\otimes L)=0$ if $x\in D\setminus(\bmx\cap\bmy)$, so the sum is on $x\in\bmx\cap\bmy$.

\begin{proof}[Proof of Lemma \ref{lem:strictincl}]
The result is local, so the setting is on a small disc with coordinate $t$ around one of the singularities of the variations of Hodge structures. The local computation mentioned above shows that there is an exact sequence
\[\tag{$*$}
0\to V^0\otimes L^0\to(V\otimes L)^0\to t^{-1}\bigoplus_{\substack{\alpha,\beta\in[0,1)\\\alpha+\beta\geq1}}\gr^\alpha V\otimes\gr^\beta L\to0.
\]
Moreover, we have $t(V\otimes L)^0\subset V^0\otimes L^0$ and
\[
(V^0\otimes L^0)/t(V\otimes L)^0\simeq\bigoplus_{\substack{\alpha,\beta\in[0,1)\\\alpha+\beta<1}}\gr^\alpha V\otimes\gr^\beta L,
\]
giving rise to the exact sequence
\begin{multline*}\tag{$**$}
0\to\bigoplus_{\substack{\alpha,\beta\in[0,1)\\\alpha+\beta<1}}\gr^\alpha V\otimes\gr^\beta L\to(V\otimes L)^0/t(V\otimes L)^0
\to t^{-1}\bigoplus_{\substack{\alpha,\beta\in[0,1)\\\alpha+\beta\geq1}}\gr^\alpha V\otimes\gr^\beta L\to0.
\end{multline*}
We also have the following Hodge filtration:
\begin{itemize}
\item
The tensor product filtration on any $\gr^\alpha V\otimes\gr^\beta L$ considered above.
\end{itemize}

For the sake of simplicity, we will set $\Gr^0V:=V^0/tV^0$ (and similarly for~$L$ and $V\otimes L$). This space is endowed with the induced filtration $F^\cbbullet\Gr^0V$. There is also a filtration $E^\cbbullet\Gr^0V$ indexed by $\alpha\in[0,1)$ induced by the decreasing $V$\nobreakdash-filtration on $\Gr^0V$, so that $\gr^\alpha_E\Gr^0V=\gr^\alpha V$. The Hodge filtration $F^\cbbullet\gr^\alpha V$ is equal to the filtration induced by $F^\cbbullet\Gr^0V$ on $\gr^\alpha_E\Gr^0V$. We have a natural morphism
\begin{equation}\label{eq:Gr0}
\Gr^0V\otimes\Gr^0L\to\Gr^0(V\otimes L)
\end{equation}
defined as follows:
\[
\Gr^0V\otimes\Gr^0L=\frac{(V^0\otimes L^0)}{t(V^0\otimes L^0)}\onto\frac{V^0\otimes L^0}{t(V\otimes L)^0}\hto\frac{(V\otimes L)^0}{t(V\otimes L)^0}=\Gr^0(V\otimes L).
\]
This morphism is compatible with the $F$-filtrations on each term. Grading with respect to $E^\cbbullet$ gives a morphism
\[
\bigoplus_{\alpha,\beta\in[0,1)}\gr^\alpha V\otimes\gr^\beta L\to\bigoplus_{\gamma\in[0,1)}\gr^\gamma(V\otimes L).
\]
The later morphism is also $F$-filtered, and is moreover a morphism of mixed Hodge structures. It is then $F$-strict. Therefore, \eqref{eq:Gr0} is also $F$-strict. Arguing similarly, we find that for any $k\geq1$ the natural morphism
\[
(V^0/t^kV^0)\otimes(L^0/t^kL^0)\to(V\otimes L)^0/t^k(V\otimes L)^0
\]
is strictly $F$-filtered.

Let us set $\wh V^0=\varprojlim_k(V^0/t^kV^0)$, endowed with $F^p\wh V^0=\varprojlim_kF^p(V^0/t^kV^0)$. We have $(\wh V^0,F^\cbbullet\wh V^0)=\wh\cO\otimes(V^0,F^\cbbullet V^0)$. The previous result implies that the inclusion
\[
\wh V^0\otimes\wh L^0\hto\wh{(V\otimes L)}{}^0
\]
is strictly $F$-filtered, hence, regarding the previous morphism as an inclusion,
\[
F^p(\wh V^0\otimes\wh L^0)=F^p\wh{(V\otimes L)}{}^0\cap (\wh V^0\otimes\wh L^0),\quad\forall p,
\]
that is,
\[
\wh\cO\otimes F^p(V^0\otimes L^0)=\wh\cO\otimes F^p(V\otimes L)^0\cap\wh\cO\otimes (V^0\otimes L^0),\quad\forall p.
\]
By faithful flatness of $\wh\cO$ over $\cO$, we conclude that
\[
F^p(V^0\otimes L^0)=F^p(V\otimes L)^0\cap(V^0\otimes L^0), \quad\forall p.\qedhere
\]
\end{proof}

\begin{proof}[Proof of Lemma \ref{lem:strictincl2}]
We consider the composed $F$-filtered morphism
\begin{equation}\label{eq:Gr02}
\Gr^1(V\otimes L)\onto\frac{t(V\otimes L)^0}{t(V^0\otimes L^0)}\hto\frac{V^0\otimes L^0}{t(V^0\otimes L^0)}=\Gr^0V\otimes\Gr^0L.
\end{equation}
After grading with respect to the $E^\cbbullet$ filtration, it becomes
\[
\bigoplus_{\gamma\in[1,2)}\gr^\gamma(V\otimes L)\to\bigoplus_{\alpha,\beta\in[0,1)}\gr^\alpha V\otimes\gr^\beta L,
\]
and has image
\[
\bigoplus_{\substack{\alpha,\beta\in(0,1)\\\alpha+\beta\geq1}}\gr^\alpha V\otimes\gr^\beta L.
\]
Being a morphism of mixed Hodge structures, it is also $F$-strict, and so is \eqref{eq:Gr02}. Since the isomorphism $t:V^0\to V^1$ is $F$-strict (and similarly for $L$ and $V\otimes L$), the isomorphism
\[
t:\frac{(V\otimes L)^0}{V^0\otimes L^0}\to\frac{t(V\otimes L)^0}{t(V^0\otimes L^0)}
\]
is also $F$-strict. As a consequence,
\[
\dim \gr^\ell_F\frac{(V\otimes L)^0}{V^0\otimes L^0}=\sum_{\substack{\alpha,\beta\in(0,1)\\\alpha+\beta\geq1}}\sum_p\dim \gr^{\ell-p}_F\gr^\alpha V\cdot\dim\gr^p_F\gr^\beta L.\qedhere
\]
\end{proof}


\section{Transformation of  local Hodge data away from $ \infty$ under middle convolution}\label{seclocdata}

Recall that in general,  the convolution $V\star L$  is neither irreducible nor  an intermediate extension 
anymore (cf. Assumption~\ref{ass1}). 
The following definition gives the largest factor
 of $V\star L$ which is an intermediate extension: 

\begin{defn}
 Let $U:=\AA^1\setminus \x \star \y $ be the smooth locus of   $V\star L.$ Define
 $   V\tstar L$ to be the intermediate extension of $V\star L|_U$ to $\AA^1.$ 
 \end{defn}
%

 For $t\in \AA^1$ let $d_t:\AA^1\to \AA^1, x\mapsto t-x,$ and write 
 $L(t-x)$ for $d_t^+L.$ 
The following result clarifies the relation between $ V\star L$ and  $V\tstar L:$

\begin{thm}\label{delta} One has  
a short exact sequence of Hodge modules
$$ 0\to V\tstar L \to V\star L \to H\to 0,$$
where 
\begin{eqnarray*}
 H&=& \left\{ \begin{array}{cl}
                        \delta_c(-p-1)  &\textrm{ if } \exists p\in \ZZ,c\in \AA^1:  V(p)\simeq   L^\vee(c-x) \\
                        0 & \mbox{ otherwise}
                       \end{array}\right. .                       
 \end{eqnarray*}
 If $H\neq 0$ then $p, c$ are uniquely determined.
 \end{thm}

\begin{proof}   As in the proof of \cite{Katz96}, Proposition~2.6.9, one finds that  
  $\cH^n(\RH(V\star_\dag L))$ vanishes outside $n=-1,0$ and that if $\cH^0(\RH(V\star_\dag L))\neq 0,$
  then there exists a unique point
  $c\in \AA$ such that $\cH^0(\RH(V\star_\dag L))$ is a rank-one skyscraper sheaf with support at $c$
  such that
  the stalk at $c$ is isomorphic to the Tate-twisted space of invariants
  $\mathrm{Hom}(\V,\mathcal{L}(c-x)^\vee)^\vee(-1).$  
  A necessary condition for the non-vanishing of  $\mathrm{Hom}(\V,\mathcal{L}(c-x)^\vee)^\vee(-1)$ 
  is that 
  one has an isomorphism of local   systems $\V\simeq \mathcal{L}(c-x)^\vee.$ 
Since irreducible VPCHS
are determined up to a Tate-twist by their local systems, there exists a unique $p$ such that the Tate-twist  $\V(p)$ 
becomes VPCHS-isomorphic to  $\mathcal{L}(c-x)^\vee$ in this case. 
This implies that, taking Hodge structures into account, 
 the stalk $\mathrm{Hom}(\V,\mathcal{L}(c-x)^\vee)^\vee(-1)$ has 
weight $p+1.$ 
Since $V\star L$ is the image of $V\star_\dag L$ and since $\cH^0(\RH(V\star_\dag L))$ maps isomorphically onto its image
inside $\RH(V\star L),$ the 
claim follows.  
 \end{proof}

\begin{thm}\label{H-numbers1}
 If $t\in \AA^1\setminus \x\star \y$
 (cf.~Def.~\ref{defconv}) 
 then   \begin{eqnarray*} 
               h^l(V\tstar  L)         
   &=& \delta^{l-1}(V \ten L(t-x))-\delta^l(V \ten L(t-x))\\
   &&  \quad \quad\quad-h^l(V \ten L(t-x))-h^{l-1}(V \ten L(t-x))+\co^{l-1}(V \ten L(t-x)) .
  \end{eqnarray*}  \end{thm}

\begin{proof} For $t\in \AA^1,$ let 
 $d_t(\y)=\{t-y_1,\ldots ,t-y_s\}$ and 
 let $j:\AA^1\setminus (\x \cup  d_t(\y)) \hookrightarrow \PP^1$ be the natural inclusion. Since $t\notin \x\star \y$ one has 
$\RH(V)\otimes\RH(L(t-x))\simeq j_*(\V\otimes \mathcal{L}(t-x))$ and hence 
$$(V\tstar L)_t=(V\star L)_t=H^1(\PP^1, j_*(\V\otimes \mathcal{L}(t-x))).$$
The claim follows now from Proposition~\ref{H-para}.  \end{proof}

The following result determines the local Hodge data of the vanishing cycles:

\begin{thm}\label{thmthomseb} 
Let $\lambda=\exp(-2\pi ia )\, (a\in (0,1])$ be a fixed element of the unit circle 
$S^1$   and let $\lambda_1,\lambda_2$ be variable elements in $S^1$ 
 with $\lambda=\lambda_1\lambda_2.$ 
  For such $\lambda_i \in S^1\, (i=1,2),$ let $a_i\in (0,1]$ with
$\lambda_i=\exp(-2\pi i a_i).$ If $t\in \x\star \y , t\neq \infty,$   then 
\begin{eqnarray*} \mu^p_{t,a}(V\star L)= \nu^p_F(\varphi_{t,\lambda}(V\star L))&=&\sum_{x_i+y_j=t} \Big( \sum_{
 a_1+a_2=a}\,\sum_{l+k=p-1}
\nu^l(\varphi_{x_i,\lambda_1}(V))\nu^k(\varphi_{y_j,\lambda_2}(L)) +  \\
& &    \quad \quad\quad  \; \; \sum_{ a_1+a_2=1+a}\,\sum_{l+k=p\,}
\nu^l(\varphi_{x_i,\lambda_1}(V))\nu^k(\varphi_{y_j,\lambda_2}(L)) 
\Big),
\end{eqnarray*}  where  the expression $\nu^p$ abbreviates $\dim{\rm gr}^p_F$. \end{thm}

\begin{proof} By Saito's version of the Thom-Sebastiani theorem
(cf.~\cite{DS}~Theorem~3.2.3 and its erratum) one knows that, for all 
$(x_i,t-y_j)$  as in  the theorem,
\begin{eqnarray} \label{eqthomseb} {\rm gr}^p_F(\varphi_{(x_i,t-y_j),\lambda}(V \boxtimes L))&=& \bigoplus_{
 a_1+a_2=a}\,\bigoplus_{l+k=p-1}
{\rm gr}^l_F(\varphi_{x_i,\lambda_1}(V))\otimes {\rm gr}^k_F(\varphi_{y_j,\lambda_2}(L)) \oplus \nonumber  \\
& &  \quad  \quad \quad\quad  \; \; \bigoplus_{ a_1+a_2=1+a}\, \bigoplus_{l+k=p\,}
{\rm gr}^l_F(\varphi_{x_i,\lambda_1}(V))\otimes {\rm gr}^k_F(\varphi_{y_j,\lambda_2}(L)) .
\end{eqnarray} 
Moreover,  the support of the vanishing cycles in the fibre over $t$ is the union of  these $(x_i,t-y_j).$ 
Since middle convolution is afterwards formed via higher direct image along 
the compactified (hence proper) $\pr_2$
and since formation of vanishing cycles is compatible with higher direct images along projective morphisms
the claim follows. 
\end{proof}

Using $\co_{\neq\infty ,a}^p(V)=\sum_{x_i\neq \infty}\co^p_{x_i,a}(V)$ and 
$ \co_{\neq \infty}^p(V)=\sum_{a} \co^p_{\neq \infty,a}(V)$ one obtains:
  
\begin{cor}\label{CorTS} Let $a\in [0,1).$ Then the following holds:
\begin{enumerate}
\item For $t\neq \infty$
\begin{eqnarray*} 
\co^{p}_{t,a}(V\star L)
 &=&  \sum_{x_i+y_j=t}\Big(\sum_{a_1+a_2=a}\sum_{l+k=p-1}
\co^{l}_{x_j,a_1}(V)\co^{k}_{y_j,a_2}(L) + \\
&&\quad \quad \quad \;
\sum_{ a_1+a_2=1+a}\sum_{l+k=p\,}
\co^l_{x_i,a_1}(V)\co^k_{y_j,a_2}(L) 
\Big) \\
\co_{\neq \infty}^p(V\star L)&=&\sum_{i+j=p}\sum_{a_1+a_2\geq 1} \co^i_{\neq \infty,a_1}(V)
\co^j_{\neq \infty,a_2}(L)+\\
&&\quad \quad \sum_{i+j=p-1}\sum_{a_1+a_2< 1} \co^i_{\neq \infty,a_1}(V)
\co^j_{\neq \infty,a_2}(L)
\end{eqnarray*}
and 
\begin{eqnarray*}  \sum_{p\leq l}\co_{\neq \infty}^p(V\star L)&=&\sum_{i+j\leq l-1} \co^i_{\neq \infty}(V)
\co^j_{\neq \infty}(L)+ \sum_{j}\sum_{a_1+a_2\geq 1} \co^j_{\neq \infty,a_1}(V)
\co^{l-j}_{\neq \infty,a_2}(L).\end{eqnarray*}
\item If $L_\chi$ is generic with respect to $L$ and $V\tstar L,$ then
$$
 h^p(V\star  (L\star  L_\chi))-h^p( (V\tstar  L)\star L_\chi)=h^p(H^0(\mathcal{H}^0(V\star L)))= \co^{p-1}_{\neq \infty} (V\star   L)                -\co^{p-1}_{\neq \infty} (V\tstar   L).$$ 
\end{enumerate}
\end{cor}

  \begin{proof} Let us first treat the case where $\lambda=1,$ equivalent to $a=0$ (note that inside Theorem~\ref{thmthomseb} the residues $a$ 
  are contained in $(0,1],$ whereas in the rest of the paper $a\in [0,1),$ hence we have to adapt our notation 
  to 
  this situation).
By Theorem~\ref{thmthomseb} 
 \begin{eqnarray*} 
 \co^{p}_{t,0}(V\star L) &=& \mu^{p+1}_{t,0}(V\star L)\\
  &=& \sum_{x_i+y_j=t} \Big( 
  \sum_{a_1+a_2=1}\sum_{l+k=p}
\nu^l(\varphi_{x_i,\lambda_1}(V))\nu^k(\varphi_{y_j,\lambda_2}(L)) \\
&& \quad \quad \quad \quad \quad \quad +
\sum_{ a_1+a_2= 2}\sum_{l+k=p+1\,}
\nu^l(\varphi_{x_i,\lambda_1}(V))\nu^k(\varphi_{y_j,\lambda_2}(L)) 
\Big)\\
 &=&  
   \sum_{x_i+y_j=t} \Big(\sum_{a_1+a_2= 0}\sum_{l+k=p+1}
\mu^l_{x_j,a_1}(V)\mu^k_{y_j,a_2}(L) \\
&&  \quad \quad \quad \quad \quad \quad +
\sum_{ a_1+a_2=1}\sum_{l+k=p}
\mu^l_{x_i,a_1}(V)\mu^k_{y_j,a_2}(L) 
\Big) \\
&=&   \sum_{x_i+y_j=t}\Big(\sum_{a_1+a_2= 0}\sum_{l+k=p+1}
\co^{l-1}_{x_j,0}(V)\co^{k-1}_{y_j,0}(L) \\
&&  \quad \quad \quad \quad \quad \quad +
\sum_{ a_1+a_2=1}\sum_{l+k=p}
\co^l_{x_i,a_1}(V)\co^k_{y_j,a_2}(L) 
\Big) \\
&=&  \sum_{x_i+y_j=t}\Big(\sum_{a_1+a_2= 0}\sum_{l+k=p-1}
\co^{l}_{x_j,0}(V)\co^{k}_{y_j,0}(L) \\
&&  \quad \quad \quad \quad \quad \quad +
\sum_{ a_1+a_2=1}\sum_{l+k=p}
\co^l_{x_i,a_1}(V)\co^k_{y_j,a_2}(L) 
\Big) .
 \end{eqnarray*}
Note that in the above sum we switch from $a_i\in (0,1]$ to $a_i\in [0,1)$ so that the case $a_1+a_2=2$
now corresponds to $a_1+a_2=0$ (and $l+k=p+1$). 
 Analogously we get for $0<a<1$
  \begin{eqnarray*} 
 \co^{p}_{t,a}(V\star L) &=& \mu^p_{t,a}(V\star L)\\
 &=&  \sum_{x_i+y_j=t}\Big(\sum_{a_1+a_2= a}\sum_{l+k=p-1}
\co^{l}_{x_j,a_1}(V)\co^{k}_{y_j,a_2}(L) \\
&&  \quad \quad \quad \quad \quad \quad +
\sum_{ a_1+a_2=1+a}\sum_{l+k=p\,}
\co^l_{x_i,a_1}(V)\co^k_{y_j,a_2}(L) 
\Big) .
   \end{eqnarray*}

   Hence the first claim follows.
   In the case where $V\star L= V\tstar L,$ the formula given in (ii) holds obviously true.   By Theorem~\ref{delta}, if $V\star L\neq V\tstar L,$ then
  \begin{equation}\label{eqtildestar}  V\star L=V\tilde{\star} L\oplus \delta_c(-q-1)\end{equation}
  and $V(q)\simeq L^\vee({c-x}).$ This implies the first equation in~(ii). Since $L_\chi$ is generic, 
  $$ (V\tstar L)\star L_\chi = (V\tstar L)\tstar L_\chi.$$ By associativity of the middle convolution (under 
  the assumption that $L,V,L_\chi$ are irreducible and nontrivial) 
  \begin{eqnarray*}   V\star  (L\star  L_\chi)&=&(V\star  L)\star  L_\chi\\
  &=&(V\tilde{\star} L\oplus \delta_c(-q-1))\star L_\chi \\
  &=& (V\tilde{\star} L)\star L_\chi \oplus L_\chi({x-c})(-q-1) .
  \end{eqnarray*}
  Therefore 
  $$ h^p(V\star  (L\star  L_\chi))-h^p( (V\tstar  L)\star L_\chi)=h^p(L_\chi({x-c})(-q-1))=\delta_{p,q+1},$$ 
  where $\delta_{i,j}$ denotes the usual Kronecker-delta. On the other hand, Rem.~\ref{remdelta} 
  and \eqref{eqtildestar} imply that        $$  \co^{p-1}_{\neq \infty} (V\star   L)    -\co^{p-1}_{\neq \infty} (V\tstar   L)= 
  \mu^p_{c,0}(V\tstar L\oplus  \delta_c(-q-1))-  \mu^p_{c,0}(V\tstar L)=\mu^p_{c,0}(\delta_c(-q-1))=\delta_{p,q+1}.$$

  \end{proof}


  \section{Transformation of  global Hodge data under middle convolution}\label{secglobal}

The following result transforms a general convolution with a Kummer Hodge module to the {\it standard situation},
considered 
in \cite{DS}, Assumption~1.1.2:

\begin{thm}
\label{Falt-Lchi}
 Let $V\star L_\chi$ be  viewed as Hodge module on $\PP^1$ by taking the minimal extension using the 
 canonical map $\AA^1_t\hookrightarrow \PP^1.$ Let $x_1,\ldots, x_r\in \AA^1(\CC)$ denote the 
 finite singularities of $V$ and let $\phi_t:\PP^1(\CC)\to \PP^1(\CC) $ be the M\"obius transformation   that exchanges $0$ and $\infty$ by inverting the 
 coordinate $t$ of $\AA^1.$  If $0$ is a smooth point 
 of $V$ then 
$V\star L_\chi$ can be obtained as
$$ V\star L_\chi =\phi_t^+((\phi_x^+(V \otimes  L_\chi)\star L_\chi) \ten L_{\chi^{-1}}).$$ 
 \end{thm} 
  
  \begin{proof}  We will view maps as $\phi_x$ or $\phi_t$ 
   either as maps on $\PP^1$ or on $\GG_m,$ or even on smaller subsets of $\GG_m,$
  depending on the context. 
Let $U:=\AA^1(\CC)\setminus \{0,x_1,\ldots,x_r\},$ let 
  $W:=U_x\times U_t\setminus \{x=t, 1=xt\}$ (with $U_x,\, U_t$ denoting copies of $U$ with coordinates $x,t$), 
   let $j:W\hookrightarrow \PP^1\times U_t$ denote the obvious embedding and let 
   $\overline{\pr}_2:\PP^1\times U_t\to U_t$ denote the second projection. 
   Then, by construction,  we have 
  $$ V\star L_\chi|_U\simeq R^1\overline{\pr}_{2+}(j_+(V(x)\otimes L_\chi(t-x))),$$
  where $V(x)$ is the Hodge module  on $W$ obtained by pulling back $V|_{U_x}$ 
  along the $x$-projection 
  and where, as above, the Hodge module  $L_\chi(t-x)$ is  obtained by pulling back
  $L_\chi$  along the map $t-x.$ We use a similar notion below also for other 
  sheaves and other polynomial maps, so that $V(f(x,t))$ denotes a pullback of $V$ along a polynomial map 
  given by $f(x,t)\in \CC[x,t].$ Then the following isomorphisms follow by looking at products of
  local sections (where the isomorphisms are seen as isomorphisms of the restrictions of Hodge modules  to  the respective 
  smooth parts):
  $$ L_\chi(t)\simeq L_{\chi^{-1}}(1/t)$$ and
   \begin{eqnarray}
   L_\chi(1/t-x)&\simeq & L_\chi(1/t)\otimes L_\chi(1-xt) \label{proof1}\\
  &\simeq & L_{\chi^{-1}}(t)\otimes L_{\chi}(x)\otimes L_\chi(1/x-t) \label{proof2}.
  \end{eqnarray}
  Let now $U':=\AA^1\setminus \{0,1/x_1,\ldots, 1/x_r\}$ 
  and consider $\phi_t:U'\to U, t\mapsto 1/t.$ Then 
  \begin{eqnarray}\quad \quad \quad  \phi_t^+(V\star L_\chi)|_{U'}&\simeq & R^1\overline{\pr}_{2+}(j_+(V(x)\otimes L_\chi(1/t-x)))|_{U'}\label{proof3}\\
  &\simeq &R^1\overline{\pr}_{+*}(j_+(V(x)\otimes L_{\chi^{-1}}(t)\otimes L_\chi(x)\otimes L_\chi(1/x-t)))|_{U'} \label{proof4}\\
  &\simeq & R^1\overline{\pr}_{2+}(\tilde{j}_+(V(1/x)\otimes L_\chi(1/x)\otimes L_{\chi^{-1}}(t)\otimes L_\chi(x-t)))|_{U'} \label{proof5}\\ 
  &\simeq &R^1\overline{\pr}_{2+}(\tilde{j}_+((V\otimes L_\chi)(1/x)\otimes L_\chi(x-t)))|_{U'}\otimes L_{\chi^{-1}}(t)|_{U'}\label{proof6}\\
  &\simeq & (\phi_x^+(V \otimes  L_\chi)\star L_\chi))|_{U'} \ten L_{\chi^{-1}}(-t)|_{U'},  \label{proof7} \end{eqnarray}
  where we use the following 
  notions and arguments: Throughout we work over the largest smooth locus of the Hodge modules involved. 
  Eq.~\eqref{proof3} holds by the discussion at the beginning of the proof. Eq.~\eqref{proof4} follows 
  from  Eq.~\eqref{proof2}. In Eq.~\eqref{proof5} we invert fibrewise the coordinate $x$ 
  and $\tilde{j}$ denotes the inclusion of the image of $U$ under this inversion to $\PP^1\times U_t.$
  Eq.~\eqref{proof6} follows from the projection formula and Eq.~\eqref{proof7} holds by definition. 
  \end{proof}


\begin{prop}\label{Lgeneric1}
Let $L_\chi$ be generic (cf.~Assumption~\ref{ass1}).
Then $V \star L_\chi=V\tstar L_\chi$ is {\rm parabolically rigid,} meaning that 
 $ H^1_{\rm par} (V \star L_\chi ) = 0.$  \end{prop}
  
  \begin{proof}
 For $t\in \AA^1\setminus \x\star \y,$
   \begin{eqnarray*}
 \rk(V\star L_\chi)&=&\co_{} (V \ten L_\chi (t-x))-2\rk(V)\\
   &=&\co_{\neq \infty} (V), 
 \end{eqnarray*}
 since $L_\chi$ is generic (cf.~\cite{DR10}, Proposition~1.2.1 and \cite{Katz96}, Corollary~3.3.7).
Hence, 
  \begin{eqnarray*}\rk(H^1_{\rm par} (V \star L_\chi ))&=&  \co_{\infty} (V\star L_\chi)+\co_{\neq \infty} (V\star L_\chi)-2 \rk(V\star L_\chi)\\
   &=&\co_{\infty} (V\star L_\chi)+
     \co_{\neq \infty} (V) -2 \rk(V\star L_\chi)\\
   &=&
     \co_{\infty} (V\star L_\chi)-\rk(V\star L_\chi)\leq 0,
\end{eqnarray*} cf.~\cite{DR10}, Proposition~1.2.1(ii), for the equality 
$\co_{\neq \infty} (V)=\co_{\neq \infty} (V\star L_\chi).$
\end{proof}
  
  The following result was independently proven using different methods by Nicolas Martin in his Dissertation   \cite{Martin}, Thm.~6.3.1:

\begin{prop}\label{Lgeneric2} 
\begin{enumerate}
\item  Let $\mu\sim 1$ (meaning that $1-\mu$ is chosen generically and small enough). Then
\begin{eqnarray*}
 h^i(V\star  L_\chi)&=& \delta^{i-1}(V) -\delta^i(V)+\co^{i-1}_{\neq \infty}(V)\\          
\delta^i(V\star  L_\chi ) &=& \delta^i(V) -\co^{i-1}_{u,\neq \infty}(V) \\
 \co^i_{\neq \infty} (V\star   L_\chi)&=& \co^{i-1}_{u,\neq \infty}(V)+
                                                             \co^{i}_{ss,\neq \infty}( V ) \\
   \co^i_{\infty} (V\star  L_\chi)&=&  h^i(V\star  L_\chi)                                                          \\
\end{eqnarray*}    
\item Let $V\star L_\chi\neq \delta_x$ for any $x\in \AA^1$ (equivalent to $V$ being not isomorphic to a translate of the dual of $L_\chi$). Then
  \[   \nu^i_{\infty,a,l}(V) = \left\{ \begin{array}{cc}
                                           \nu^{i+1}_{\infty,1-\mu,l+1}(V\star L_\chi), & 0=a   \\
                                         \nu^{i}_{\infty,a+1-\mu,l}(V\star L_\chi), & 0< a<\mu   \\
                                           \nu^{i}_{\infty,0,l-1}(V\star L_\chi), & a=\mu, l>0   \\
                                         \nu^{i+1}_{\infty,a-\mu,l}(V\star L_\chi), & a>\mu   \\
                                      \end{array} \right. .\]
                                      Moreover,  the only other possibly non zero nearby cycle data at infinity of $V\star L_\chi$  are of the form $\nu^i_{\infty,1-\mu,0}(V\star L_\chi)$.
                                      If $\mu\sim 1$ then the formula simplifies to   \[   \nu^i_{\infty,a,l}(V) = \left\{ \begin{array}{cc}
                                         \nu^{i+1}_{\infty,1-\mu,l+1}(V\star L_\chi), & a=0   \\
                                          \nu^{i}_{\infty,a+1-\mu,l}(V\star L_\chi), & a\neq 0   \\
                                      \end{array} \right. .\]\end{enumerate}
\end{prop}

 \begin{proof}    
 By Theorem~\ref{H-numbers1} and Theorem~\ref{H-tensor},
  \begin{eqnarray*}
 h^i(V\star  L_\chi)&=& h^i(V\tstar L_\chi)\\
 &=& \delta^{i-1}(V\otimes L_\chi(t-x)) -\delta^i(V\otimes L_\chi(t-x))-h^i(V\otimes L_\chi(t-x))\\
 && \quad \quad \quad \quad \quad \quad\quad \quad \quad -h^{i-1}(V\otimes L_\chi(t-x))
 +\co^{i-1}(V\otimes L_\chi(t-x))\\  
 &=& (\delta^{i-1}(V)-h^{i-1}(V)+0)-(\delta^i(V)-h^i(V)+0)-h^i(V)\\
 && \quad \quad \quad \quad \quad \quad\quad \quad \quad -h^{i-1}(V) +\co^{i-1}(V\otimes L_\chi(t-x))\\
&=& \delta^{i-1}(V)-h^{i-1}(V)-\delta^i(V)\\
 &&-h^{i-1}(V) +
 \co^{i-1}_{\neq \infty, t}(V\otimes L_\chi(t-x))+ \co^{i-1}_{t}(V\otimes L_\chi(t-x)) +
 \co^{i-1}_{\infty}(V\otimes L_\chi(t-x))\\
 &=&  \delta^{i-1}(V)-h^{i-1}(V)-\delta^i(V)\\
 && \quad \quad \quad \quad \quad \quad\quad \quad \quad -h^{i-1}(V) +
 \co^{i-1}_{\neq \infty}(V)+ h^{i-1}(V) +
 h^{i-1}(V)\\
 &=&\delta^{i-1}(V) -\delta^i(V)+\co^{i-1}_{\neq \infty}(V),
 \end{eqnarray*}  
which is the first formula in~(i).   

The fixed space under the local monodromy at $\infty$ is trivial
(otherwise,  the last formula in the proof of Proposition~\ref{Lgeneric1} reads $\rk(H^1_{\rm par}(V \star L_\chi ))<0,$
a contradiction).
This implies $\co^i_{\infty} (V\star  L_\chi)= h^i(V\star  L_\chi)$ which is the fourth equation in~(i). By Proposition~\ref{Lgeneric1} $V\star L_\chi$ is parabolically rigid. Therefore,
\begin{eqnarray*}
   0&=& h^i(H^1_{\rm par }(V \star L_\chi ))\\  
   &=&\delta^{i-1}(V \star L_\chi )-\delta^{i}(V \star L_\chi )-h^{i}(V \star L_\chi )-h^{i-1}(V \star L_\chi )+
  \co^{i-1}(V\star L_\chi) \\
     &=&\delta^{i-1}(V \star L_\chi )-\delta^{i}(V \star L_\chi )-h^{i}(V \star L_\chi )+
  \co^{i-1}_{\neq \infty}(V\star L_\chi). \\
 \end{eqnarray*} 
 Therefore
 \begin{eqnarray}\label{eq331}
 \delta^{i}(V \star L_\chi )-\delta^{i-1}(V \star L_\chi ) &=&-h^{i}(V \star L_\chi )+
  \co^{i-1}_{\neq \infty}(V\star L_\chi).\end{eqnarray}
  Using $\mu\sim 1$ and  the second equality in Corollary~\ref{CorTS}(i) we obtain 
  \begin{eqnarray*}
\co^i_{\neq \infty} (V\star   L_\chi)&=& \co^{i-1}_{u,\neq \infty}(V)+
                                                             \co^{i}_{ss,\neq \infty}( V ),\end{eqnarray*}
                                                             establishing the third equality in~(i).
Hence
   \begin{eqnarray*}
 \delta^{i}(V \star L_\chi )-\delta^{i-1}(V \star L_\chi ) &=&-h^{i}(V \star L_\chi )+
  \co^{i-1}_{\neq \infty}(V\star L_\chi)\\
  &=&\delta^i(V)-\delta^{i-1}(V)-\co^{i-1}_{\neq \infty}(V)+\co^{i-1}_{ss,\neq \infty}(V) +\co^{i-2}_{u,\neq \infty}(V)\\
   &=&\delta^i(V)-\delta^{i-1}(V)-\co^{i-1}_{u,\neq \infty}(V) +\co^{i-2}_{u,\neq \infty}(V).\\
 \end{eqnarray*}   Summing up we get the second formula of~(i):
 \begin{eqnarray*}
   \delta^{i}(V \star L_\chi )&=&\delta^i(V)-\co^{i-1}_{u,\neq \infty}(V).
 \end{eqnarray*}  
 
Claim~(ii) follows by rewriting $V \star L_\chi$ via Theorem~\ref{Falt-Lchi} as
 \begin{eqnarray}\label{eq31}
  V \star L_\chi&=&\phi_t^+((\phi_x^+(V \otimes  L_\chi)\star L_\chi) \ten L_{\chi^{-1}})
 \end{eqnarray}
 assuming via a translation that $0$ is not a singularity of $V$. 
 Thus the singularity $\infty$ of $V$ becomes the finite singularity $0$ of $\phi_x^+(V \otimes  L_\chi)$ and we can apply Formula~\eqref{eqthomseb} in order to prove the first claim.
 Note that $\phi_x^+(V \otimes  L_\chi)$ has scalar monodromy $\exp(-2\pi i \mu)$ at $\infty$, which is called the standard situation for the middle
 convolution with the Kummer sheaf $L_\chi$ in \cite{DS}. The first Formula in~(ii) follows now from 
 \cite{DS},~Theorem~3.1.2(2).
  By Formula~\eqref{eq31} and Theorem~\ref{thmthomseb},  the only other possibly non zero nearby cycle data at infinity  are of the form
  $\nu^i_{\infty,1-\mu,0}(V\star L_\chi)$.
   \end{proof}

\begin{thm}\label{deltaFalt}
Let $V\tstar L\neq 0$.
Then
  \begin{eqnarray*} \delta^l(V \tstar  L)&=&
       \sum_{i}^{}   (\co^i_{\neq \infty}(V) \delta^{l-1-i}(L) + \delta^i(V) \co^{l-1-i}_{\neq \infty} (L) )
        +\sum_{j}^{} \delta^{j}(V) \delta^{l-1-j}(L)- \sum_{j} \delta^{j}(V) \delta^{l-j}(L) \\
&& +\sum_{i}^{} \sum_{a+b\geq 1}  \co^i_{\neq \infty,a}(V) \co^{l-1-i}_{\neq \infty,b}(L)+\sum_{j}^{} \sum_{a+b\geq 1} \co^j_{\infty,a}(V) \co^{l-j}_{\infty,b}(L). 
  \end{eqnarray*}
\end{thm}

\begin{proof}
 Let $L_\chi$ be generic and $\mu \sim 1$.
 By Proposition~\ref{Lgeneric2}(i)
 \begin{eqnarray*}
   \sum_{i\leq l}^{} h^i((V\tstar L) \star L_\chi) 
                                      &=& -\delta^l(V\tstar L)+\sum_{i\leq l}^{} \co^{i-1}_{\neq \infty}(V\tstar L).
 \end{eqnarray*}

 By the transformation of residues under convolution, described by Proposition~\ref{Lgeneric2}(ii), 
the
sum of a residue  of $L\star L_\chi$ at $\infty$ and a residue of $V$ at $\infty$ is not an integer.
Hence the nearby cycles of $V\ten (L\star L_\chi)(t-x)$ at $\infty$
coincide with the vanishing cycles.
Therefore
$$   \co^{i-1}_{\infty}(V\ten (L\star L_\chi)(t-x))= h^{i-1}(V\ten (L\star L_\chi)(t-x)).$$
This and  Theorem~\ref{H-numbers1} imply
\begin{eqnarray}\label{hi}
 h^i(V\tstar (L \star L_\chi))
   &=&        \delta^{i-1}(V\ten (L\star L_\chi)(t-x))-\delta^{i}(V\ten (L\star L_\chi)(t-x))-h^{i}(V\ten (L\star L_\chi)(t-x)) \nonumber \\
 &&-h^{i-1}(V\ten (L\star L_\chi)(t-x))    +         \co^{i-1}_{}(V\ten (L\star L_\chi)(t-x))   \nonumber \\
&=&   
           \delta^{i-1}(V\ten (L\star L_\chi)(t-x))-\delta^{i}(V\ten (L\star L_\chi)(t-x))-h^{i}(V\ten (L\star L_\chi)(t-x))+
             \\
 &&                                          \co^{i-1}_{\neq \infty}(V\ten (L\star L_\chi)(t-x)). \nonumber
\end{eqnarray}

One has 
\begin{eqnarray}\label{hi2} &&h^{i}(V\ten (L\star L_\chi)(t-x)) -\co^{i-1}_{\neq \infty}(V\ten (L\star L_\chi)(t-x)) \nonumber \\
 &=&  \sum_{j}^{} h^j(V)  h^{i-j}((L\star L_\chi)(t-x))-\co^{i-1}_{\neq \infty}(V\ten (L\star L_\chi)(t-x))  \nonumber \\
   &=&   \sum_{j}^{} h^j(V)  h^{i-j}((L\star L_\chi)(t-x))
    -  \sum_{j}^{} h^{j}(V) \co^{i-1-j}_{\neq \infty}((L\star L_\chi)(t-x))
  -\sum_{j}^{}\co^{j}_{\neq \infty}(V)h^{i-j-1}(L\star L_\chi) 
 \nonumber  \\
  &=&  \sum_{j}^{} h^j(V) (h^{i-j}((L\star L_\chi)(t-x))-\co^{i-j-1}_{\neq \infty}((L\star L_\chi)(t-x)))
  -\sum_{j}^{}\co^{j}_{\neq \infty}(V)h^{i-j-1}(L\star L_\chi) 
 \nonumber  \\
  &=& \sum_{j}^{} h^j(V)(\delta^{i-j-1}(L)-\delta^{i-j}(L)+
                  \co^{i-j-1}_{\neq \infty, u}(L)-\co^{i-j-2}_{\neq \infty, u}(L))\nonumber  \\
     && \quad \quad \quad    \quad \quad \quad \quad \quad \quad    -\sum_{j}^{}\co^{j}_{\neq \infty}(V)  (\delta^{i-j-2}(L)- \delta^{i-j-1}(L)+\co^{i-j-2}_{\neq \infty}(L)) ,     
  \end{eqnarray}
where we used the following arguments:
the first equality uses the usual equality for tensor products,
the second follows from  basic properties of nearby cycles of tensor products,
the third equality is a reorganisation of the sum and in the last equation we use
the first and the third equation in
 Proposition~\ref{Lgeneric2}(i).

Summing up \eqref{hi}  yields
\begin{eqnarray}\label{hi3}
  \sum_{i\leq l}^{} h^i(V\tstar (L \star L_\chi)) &=&-\delta^{l}(V\ten (L\star L_\chi)(t-x))  
  +\sum_{i }^{} h^{i}(V)(\delta^{l-i}(L)-\co^{l-1-i}_{\neq \infty, u}(L))  \nonumber   \\ 
    &&            - \sum_{i}^{}  \co^{i}_{\neq \infty}(V) \delta^{l-i-1}(L)+
    \sum_{i+k\leq l}^{} \co^{k}_{\neq \infty}(V)\co^{i-k-2}_{\neq \infty}(L) \nonumber   \\
   &=&  -\delta^{l}(V\ten (L\star L_\chi)(t-x))  
  +\sum_{i }^{} h^{i}(V)(\delta^{l-i}(L)-\co^{l-1-i}_{\neq \infty, u}(L))  \nonumber   \\ 
    &&      - \sum_{i}^{}  \co^{i}_{\neq \infty}(V) \delta^{l-i-1}(L)+
        \sum_{i\leq l-1}\co^{i}_{\neq \infty}(V \star L)  \nonumber \\
&&      -  \sum_{j} \sum_{a+b\geq 1} \co^{j}_{\neq \infty,a}(V) \co^{l-1-j}_{\neq \infty,b}( L),
\end{eqnarray} 
where we use  \eqref{hi2} for the first equality and 
Corollary~\ref{CorTS}(i) for the second.

On the other hand, by the first equality in Proposition~\ref{Lgeneric2}(i),
 \begin{eqnarray*}
   \sum_{i\leq l}^{} h^i((V\tstar L) \star L_\chi) 
                                      &=& -\delta^l(V\tstar L)+\sum_{i\leq l}^{} \co^{i-1}_{\neq \infty}(V\tstar L).
 \end{eqnarray*}
 
Thus by Corollary~\ref{CorTS}(ii) and Eq.~\eqref{hi3}
\begin{eqnarray*}
0&=&  \sum_{i\leq l}^{} h^i(V\star (L \tstar L_\chi))- \sum_{i\leq l}^{} h^i((V\tstar L) \star L_\chi)-
     \sum_{i\leq l-1}\co^{i}_{\neq \infty}(V \star L)+ \sum_{i\leq l-1}\co^{i}_{\neq \infty}(V \tstar L)\\
     &=&  -\delta^{l}(V\ten (L\star L_\chi)(t-x))  
  +\sum_{i }^{} h^{i}(V)(\delta^{l-i}(L)-\co^{l-1-i}_{\neq \infty, u}(L))    
        - \sum_{i}^{}  \co^{i}_{\neq \infty}(V) \delta^{l-i-1}(L)  \\
&&   +
        \sum_{i\leq l-1}\co^{i}_{\neq \infty}(V \star L)   -  \sum_{j} \sum_{a+b\geq 1} \co^{j}_{\neq \infty,a}(V) \co^{l-1-j}_{\neq \infty,b}( L)\\
&& + \,\delta^l(V\tstar L) -\sum_{i\leq l}^{} \co^{i-1}_{\neq \infty}(V\tstar L) -\sum_{i\leq l-1}\co^{i}_{\neq \infty}(V \star L)+\sum_{i\leq l-1}\co^{i}_{\neq \infty}(V \tstar L)\\
   &=& \delta^l(V\tstar L)-\delta^{l}(V\ten (L\star L_\chi)(t-x))+\sum_{i}^{} h^{i}(V)(\delta^{l-i}(L)-\co^{l-1-i}_{\neq \infty, u}(L))
                       -\sum_{i}^{}  \co^{i}_{\neq \infty}(V) \delta^{l-i-1}(L)-\\
&& \sum_{i}^{} \sum_{a+b\geq 1} \co^i_{\neq \infty,a}(V) \co^{l-1-i}_{\neq \infty,b}(L).
      \end{eqnarray*}

By Theorem~\ref{H-tensor} and Proposition~\ref{Lgeneric2},
\begin{eqnarray*}
 \delta^{l}(V\ten (L\star L_\chi)(t-x))&=&\sum_{i}^{} (\delta^i(V) h^{l-i}(L\star L_\chi) +h^i(V) \delta^{l-i}(L\star L_\chi))+o^l_\infty(V\ten (L\star L_\chi))\\
 &=&\sum_{i}^{} \delta^i(V) (\delta^{l-i-1}(L)-\delta^{l-i}(L)+\co^{l-i-1}_{\neq \infty}(L))  +\sum_{i}^{} h^i(V) (\delta^{l-i}(L)-
               \co^{l-i-1}_{\neq \infty,u}(L))\\
               &&+\sum_{j}^{} \sum_{a+b\geq 1}  \co^j_{\infty,a}(V)\co^{l-j}_{\infty,b}(L).
\end{eqnarray*}
Taking the previous two equations together one has 
\begin{eqnarray*}\delta^l(V\tstar L)&=&\sum_{i}^{} \delta^i(V) (\delta^{l-i-1}(L)-\delta^{l-i}(L)+\co^{l-i-1}_{\neq \infty}(L))  +\sum_{i}^{} h^i(V) (\delta^{l-i}(L)-
               \co^{l-i-1}_{\neq \infty,u}(L))\\
               &&-\sum_{j}^{} \sum_{a+b\geq 1}  \co^j_{\infty,a}(V)\co^{l-j}_{\infty,b}(L)
 -\sum_{i}^{} h^{i}(V)(\delta^{l-i}(L)-\co^{l-1-i}_{\neq \infty, u}(L))
                       +\sum_{i}^{}  \co^{i}_{\neq \infty}(V) \delta^{l-i-1}(L)\\
&&+ \sum_{i}^{} \sum_{a+b\geq 1} \co^i_{\neq \infty,a}(V) \co^{l-1-i}_{\neq \infty,b}(L)\\
&=&
       \sum_{i}^{}   (\co^i_{\neq \infty}(V) \delta^{l-1-i}(L) + \delta^i(V) \co^{l-1-i}_{\neq \infty} (L) )  +\sum_{i}^{} \delta^{j}(V) \delta^{l-1-i}(L)- \sum_{i} \delta^{i}(V) \delta^{l-i}(L) \\
&& +\sum_{i}^{} \sum_{a+b\geq 1}  \co^i_{\neq \infty,a}(V) \co^{l-1-i}_{\neq \infty,b}(L)+\sum_{i}^{} \sum_{a+b\geq 1} \co^i_{\infty,a}(V) \co^{l-i}_{\infty,b}(L),
\end{eqnarray*}
as claimed.
\end{proof}


\section{Transformation of  local Hodge data at $ \infty$ under middle convolution}\label{seclocalinfty}

In the following, the objects $V,L,M$ satisfy the  conditions in Assumption~\ref{ass1}. 
  
  \begin{thm}\label{cohoFalt}
   Let $\epsilon_l:=h^{l+1} ({H^0(\mathcal H}^0(\RH(V \star L)))).$
   Then
   $$     h^{l+1}(H^1_{\rm par} (V \tstar L)) +\kappa^l_\infty(V\tstar L)+\epsilon_l-\kappa^l_\infty(V\ten L (t-x))=$$
$$    \sum_{i}^{} (h^i(H^1_{\rm par}(V))+\kappa^{i-1}_\infty(V)) (h^{l+1-i}(H^1_{\rm par}(L))+\kappa^{l-i}_\infty(L)  ) $$
    
  \end{thm}

  \begin{proof}
  By Theorem~\ref{H-para}
     \begin{eqnarray*}
    h^{l+1}(H^1_{\rm par} (V \tstar L))&=&  
     \delta^l(V \tstar L) - \delta^{l+1}(V \tstar L) -h^{l+1}(V\tstar L)-\kappa^l_\infty(V\tstar L) +\co^l_{\neq \infty} (V \tstar L).
    \end{eqnarray*}
  Using subsequently Theorem~\ref{H-numbers1}, Theorem~\ref{H-tensor} and Theorem~\ref{H-para} we get
   \begin{eqnarray*}
    h^{l+1}(V\tstar L) &=&  \delta^l(V \ten L (t-x)) - \delta^{l+1}(V \ten L(t-x) ) -h^{l+1}(V\ten L(t-x))-\kappa^l_\infty(V\ten L (t-x)) \\
                          &&+\co^l_{\neq \infty} (V \ten L(t-x))\\
                 &=& \sum_{i }^{} (\delta^{i-1} (V)-\delta^i(V)-h^{i}(V)+\co^{i-1}_{\neq \infty}(V))  h^{l+1-i}(L) \\
                 &&   + \sum_{i}^{} h^i(V)(\delta^{l-i}(L)-\delta^{l+1-i}(L)+ \co^{l-i}_{\neq \infty}(L)) \\
                   && +o^l_\infty(V\ten L (t-x))-o^{l+1}_\infty(V\ten L (t-x))-\kappa^l_\infty(V\ten L (t-x)) \\
                  &=&    \sum_{i}^{} (h^{i}(H^1_{\rm par}(V))+\kappa^{i-1}_\infty(V)) h^{l+1-i}(L) \\
                     &&   + \sum_{i}^{} h^i(V)(\delta^{l-i}(L)-\delta^{l+1-i}(L)+ \co^{l-i}_{\neq \infty}(L)) \\
                       && +o^l_\infty(V\ten L (t-x))-o^{l+1}_\infty(V\ten L (t-x))-\kappa^l_\infty(V\ten L (t-x)). 
   \end{eqnarray*}
By Theorem~\ref{deltaFalt}
    \begin{eqnarray*}
     \delta^{l}(V\tstar L) &=&  \sum_{i}^{} (h^i(V)+\kappa^{i-1}_\infty(V)+h^i(H^1_{\rm par}(V)) )\delta^{l-i}(L) +
                            \sum_{i}^{} \delta^{i-1}(V) \co^{l-i}_{\neq \infty}(L)\\
                            && +  o^l_\infty(V\ten L (t-x))+\sum_{i}^{} \sum_{a+b\geq 1}  \co^i_{\neq \infty,a}(V)\co^{l-1-i}_{\neq \infty,b}(L).\\
    \end{eqnarray*}
Inserting the last two equations into the first equation we obtain

  \begin{eqnarray*}
    h^{l+1}(H^1_{\rm par} (V \tstar L))) &=&  
            \sum_{i}^{} (\kappa^{i-1}_\infty(V)+h^i(H^1_{\rm par}(V)) )(\delta^{l-i}(L)-\delta^{l+1-i}(L)+h^{l+1-i}(L))+ \\
                         &&\\
                          &&  \sum_{i}^{} (\delta^{i-1}(V)-\delta^i(V)-h^i(V)) \co^{l-i}_{\neq \infty}(L)+\\
                          &&\\
                       && \kappa^l_\infty(V\ten L (t-x)) -\kappa^l_\infty(V\tstar L) +\co^l_{\neq \infty} (V \tstar L)\\
                       && -\sum_{i}^{} \sum_{a+b\geq 1}  \co^i_{\neq \infty,a}(V) \co^{l-i}_{\neq \infty,b}(L)+ \sum_{i}^{} \sum_{a+b\geq 1} \co^i_{\neq \infty,a}(V)\co^{l-1-i}_{\neq \infty,b}(L)
                     .\\
   \end{eqnarray*}
By Corollary~\ref{CorTS}(ii),
\[ \co^l_{\neq \infty} (V \star L)-\co^l_{\neq \infty} (V \tstar L)=\epsilon_l. \]
Moreover, by Corollary~\ref{CorTS}(i)
$$ \co_{\neq \infty}^l(V\star L)=\sum_{i+j=l}\sum_{a_1+a_2\geq 1} \co^i_{\neq \infty,a_1}(V)
\co^j_{\neq \infty,a_2}(L)+
\sum_{i+j=l-1}\sum_{a_1+a_2< 1} \co^i_{\neq \infty,a_1}(V)
\co^j_{\neq \infty,a_2}(L) $$
Inserting this into the above equation yields \begin{eqnarray*}
     h^{l+1}(H^1_{\rm par} (V \tstar L))&=&  
             \sum_{i}^{} (\kappa^{i-1}_\infty(V)+h^i(H^1_{\rm par}(V)) )(\delta^{l-i}(L)-\delta^{l+1-i}(L)+h^{l+1-i}(L))+ \\
                         &&\\
                          &&  \sum_{i}^{} (\delta^{i-1}(V)-\delta^i(V)-h^i(V)) \co^{l-i}_{\neq \infty}(L)+\\
                          &&\\
                       && \kappa^l_\infty(V\ten L (t-x))  -\kappa^l_\infty(V\tstar L) -\epsilon_l+\sum_{i}^{} \co^{i-1}_{\neq \infty}(V) \co^{l-i}_{\neq \infty}(L)\\    
                       &=&    
             \sum_{i}^{} (\kappa^{i-1}_\infty(V)+h^i(H^1_{\rm par}(V)) )(\delta^{l-i}(L)-\delta^{l+1-i}(L)+h^{l+1-i}(L))+ \\
                         &&\\
                          &&  \sum_{i}^{} (\delta^{i-1}(V)-\delta^i(V)-h^i(V)+\co^{i-1}_{\neq \infty}(V) ) \co^{l-i}_{\neq \infty}(L)+\\
                          &&\\
                       && \kappa^l_\infty(V\ten L (t-x))  -\kappa^l_\infty(V\tstar L) -\epsilon_l\\
                       &=&
             \sum_{i}^{} (\kappa^{i-1}_\infty(V)+h^i(H^1_{\rm par}(V)) )(\delta^{l-i}(L)-\delta^{l+1-i}(L)+h^{l+1-i}(L))+ \\
                         &&\\
                          &&  \sum_{i}^{} (\kappa^{i-1}_\infty(V)+h^i(H^1_{\rm par}(V)) ) \co^{l-i}_{\neq \infty}(L)+\kappa^l_\infty(V\ten L (t-x))  -\kappa^l_\infty(V\tstar L) -\epsilon_l\\
                  &=&       \sum_{i}^{} (h^i(H^1_{\rm par}(V))+ \kappa^{i-1}_\infty(V)) (\delta^{l-i}(L)-\delta^{l+1-i}(L)-h^{l+1-i}(L)+ \co^{l-i}_{\neq \infty}(L))\\
                          &&\\
                       &&+ \kappa^l_\infty(V\ten L (t-x))-\kappa^l_\infty(V\tstar L) -\epsilon_l\\   
            &=&  \sum_{i}^{} (h^i(H^1_{\rm par}(V))+\kappa^{i-1}_\infty(V)) (h^{l+1-i}(H^1_{\rm par}(L))+\kappa^{l-i}_\infty(L)  )  \\
                       && +\kappa^l_\infty(V\ten L (t-x)) -\kappa^l_\infty(V\tstar L) -\epsilon_l,\\
  \end{eqnarray*}             
  as claimed. 
 \end{proof}

   The following result is also obtained  in the Dissertation of Nicolas Martin  \cite{Martin}, Thm.~6.3.1:

 \begin{cor}\label{corcohofalt} Let $V\star L_\chi\neq \delta_x$ for any $x\in \AA^1.$ 
Then  $$\nu^i_{\infty,1-\mu,0}(V\star L_\chi)=h^i(H^1_{\rm par}(V)).$$
 \end{cor}
 
 \begin{proof}    Since  $V\tstar L_\chi=V\star L_\chi$ we have
  $$ (V\star L_\chi) \tstar L_{\overline{\chi}} = V(-1).$$
 Hence, since $L_\chi$ is parabolically rigid, it follows from Theorem~\ref{cohoFalt} that 
 $$ h^{i+1}(H^1_{\rm par}(V(-1)) )+ \kappa^{i}_\infty(V(-1)) -\kappa^{i}_\infty((V \star L_\chi) \ten L_{\overline{\chi}} (t-x)))=0.$$
 By definition of $\kappa_\infty$ and the first formula in  Proposition~\ref{Lgeneric2}~(ii)
 \begin{eqnarray*}
 \kappa^{i}_\infty((V \star L_\chi) \ten L_{\overline{\chi}} (t-x))&=&
  \nu^i_{\infty,1-\mu,{\rm prim}}(V \star L_\chi)\\
  &=&\nu^i_{\infty,1-\mu,0}(V \star L_\chi)+\nu^{i-1}_{\infty,0,{\rm prim}}(V)\\
  &=&
  \nu^{i}_{\infty,1-\mu,0}(V \star L_\chi)+\kappa^{i}_\infty(V(-1)).\end{eqnarray*} 
  Hence we obtain
  \[ \nu^i_{\infty,1-\mu,0}(V\star L_\chi)=h^{i+1}(H^1_{\rm par}(V(-1)))=h^i(H^1_{\rm par}(V)).\]

 \end{proof}

 \begin{rem}
 The above theorem may also be derived from the well known formula
 $$ H^*(\AA^1,V\star_* L)=H^*(\AA^1,V)\otimes H^*(\AA^1,L)$$ 
 with 
 $$ h^{l+1}(H^*(\AA^1,V)\otimes H^*(\AA^1,L))=    \sum_{i}^{} (h^i(H^1_{\rm par}(V))+\kappa^{i-1}_\infty(V)) (h^{l+1-i}(H^1_{\rm par}(L))+\kappa^{l-i}_\infty(L)  )$$
 and
 $$h^{l+1}(H^*(\AA^1,V\star_* L))=h^{l+1}(H^1_{\rm par} (V \tstar L)) +\kappa^l_\infty(V\tstar L)+\epsilon_l-\kappa^l_\infty(V\ten L (t-x)).$$ 
 In the last equation one uses the usual long exact cohomology sequence and 
 $$H^0(\AA^1,V\star_* L)=H^0(\mathcal{H}^0(\RH(V\star L))),$$ 
 cf.~\cite{Katz96},  Lemma~2.6.9.
 
 \end{rem}

 \begin{rem}\label{tstar}
 If $V\star L=V\tstar L, L\star M=L\tstar M$ and $(V\tstar L)\tstar M = (V\star L) \star M$ then
  \begin{eqnarray*}
  (V\tstar L)\tstar M 
                     &=& V\tstar (L\tstar M).\\ 
  \end{eqnarray*}
 \end{rem}
 
 \begin{proof}
 By assumption we have
 \begin{eqnarray*}
  (V\tstar L)\tstar M &=& (V\tstar L)\star M\\ 
                 &=& (V\star L)\star M\\ 
                  &=& V\star (L\star M)\\ 
                   &=& V\star (L\tstar M),
   \end{eqnarray*}
   using the associativity of the middle convolution under the Assumption~\ref{ass1} (cf.~\cite{Katz96}, Lemma~2.6.5). 
  By the nature of the $\tstar$-convolution, in the term on the  the left hand side of the previous 
   equation, there appears no skyscraper sheaf $\delta_c$ as a summand. 
  Hence we can conclude that
  $V\star (L\tstar M)=V\tstar (L\tstar M)$ and finally
  \begin{eqnarray*}
  (V\tstar L)\tstar M &=&V\tstar (L\tstar M).
  \end{eqnarray*}
 \end{proof}

 \begin{cor}\label{KunF}
  Let $V,L,M$ be parabolically rigid without unipotent Jordan blocks at $\infty$.
  Assume further that $V \ten L(t-x)$ has also no unipotent Jordan block
  at $\infty$. 
 \begin{enumerate}
  \item Then $V\star L=V\tstar L$ is also parabolically rigid (not necessarily irreducible) without unipotent Jordan blocks at $\infty$.
\item  Moreover if  $L\star M=L\tstar M$ and $(V\tstar L)\tstar M=(V\tstar L)\star M$ then
 
 \begin{eqnarray*}
 \kappa^l_\infty((V\tstar L) \ten M (t-x)) &=&  \kappa^l_\infty(V \ten (L \tstar M)(t-x)). \\
  \end{eqnarray*}
  \end{enumerate}
  
 \end{cor}

\begin{proof}
We have $\kappa^l_\infty(L)=\kappa^l_\infty(L)=\kappa^l_\infty(M)=0$ since there is no unipotent Jordan block at $\infty$.
The assumption that  $V \ten L(t-x)$ has  no unipotent Jordan block
  at $\infty$ implies that $V$ is not dual to a translate of the form $L({c-x})$. 
  Hence $V\star L=V\tstar L$ by Theorem~\ref{delta}.
  Further $V\tstar L$ is parabolically rigid  without unipotent Jordan block at $\infty$ by Theorem~\ref{cohoFalt},
  implying~(i).

   By Theorem~\ref{cohoFalt}, 
  $$     h^{l+1}(H^1_{\rm par } ((V \tstar L) \tstar M))) +\kappa^l_\infty((V\tstar L)\tstar M) +\epsilon_l((V\tstar L)\tstar M)
    -\kappa^l_\infty((V\tstar L)\ten M (t-x))=0$$
    and 
$$  h^{l+1}(H^1_{\rm par} (V \tstar (L \tstar M))) +\kappa^l_\infty(V\tstar (L \tstar M)) +\epsilon_l(V\tstar (L\tstar M))
    -\kappa^l_\infty((V \ten (L\tstar M) (t-x))=0.$$
   Since by Remark~\ref{tstar}~(ii), $(V\tstar L)\tstar M= V\tstar (L\tstar M)$ we deduce
  \begin{eqnarray*}
  \kappa^l_\infty((V\tstar L) \ten M (t-x)) &=&  \kappa^l_\infty(V \ten (L \tstar M)(t-x)). \\
  \end{eqnarray*}
  
  \end{proof}

 \begin{rem}\label{Jordandec} 
 Let $m,n\in \NN_{>0}$ and let $a_m,b_n\in \RR\cap (0,1).$

 \begin{enumerate}
  \item    Let $V,L=L(n,b_n)$ be as in Assumption~\ref{ass1} such that 
      $\psi_\infty(V)\simeq J^{m-1}(a_m,m)$ and $\psi_\infty(L) \simeq J^{n-1}(b_n,n)$. Then the  non zero Hodge numbers of $V\otimes L$ are
   $$ h^p(V\ten L) =\left\{ \begin{array}{cccc}
                                                     p+1, & 0\leq p\leq  \min\{m,n\}-1 \\
                                                     \min\{n,m\}, & \min\{m,n\} \leq p< m+n-\min\{m,n\} \\
                                                     m+n-p-1, &  m+n-\min\{m,n\}\leq p<m+n
                                                     \end{array}\right.
 $$
 and $ \psi_\infty(V \ten L)$ is isomorphic to 
 $$
      J^{m+n-2}(a_m+b_n-[a_m+b_n ],m+n-1) \oplus \dots \oplus J^{m+n-1-\min\{m,n\}}(a_m+b_n-[a_m+b_n ],m+n+1-2\min\{m,n\}).$$

 \item  Let $V,L=L(n,b_n)$ be as in Assumption~\ref{ass1}.  Then the  structure of 
 $ \psi_\infty(V) $ is uniquely determined by
  $ \kappa^l_\infty(V\ten L(n,b_n))$ for all $n, b_n.$ 
 \end{enumerate}
 \end{rem}

 \begin{proof}
  The tensor decomposition of $J(n) \ten J(m)$ of
    the tensor product of two unipotent Jordan blocks of size $n$, resp. $m$, in characteristic zero is given by
   $$J(m+n-1) \oplus J(m+n-3) \oplus \ldots \oplus J(m+n+1-2\min\{m,n\}),$$
   cf. Reference Chapter, Table 5, $A_1$,\cite{OnischikVinberg}.
  Moreover, if
   $\psi_\infty(V)\simeq J^{m-1}(a_m,m)$ and $\psi_\infty(L)\simeq J^{n-1}(b_n,n)$ then
   $$ h^0(V)=\ldots=h^{m-1}(V)=1,\quad  h^0(L)=\ldots=h^{n-1}(L)=1$$
   and the non zero Hodge numbers are
   $$ h^p(V\ten L)= \sum_{i+j=p} h^i(V) h^j(L) =\left\{ \begin{array}{cccc}
                                                     p+1, & 0\leq p\leq  \min\{m,n\}-1 \\
                                                     \min\{n,m\}, & \min\{m,n\} \leq p< m+n-\min\{m,n\} \\
                                                     m+n-p-1, &  m+n-\min\{m,n\}\leq p<m+n
                                                     \end{array}\right..
 $$
Since
$$ \# \{ p \mid h^p(V\ten L) \geq 1 \} =m+n-1$$
and $\nu^{i_1}_{\infty,a_m+b_n-[a_m+b_n],m+n-2}(V\ten L))\geq 1$ for some $i_1$ we obtain $i_1=m+n-2$. 
Since
$$ \# \{ p \mid h^p(V\ten L)  \geq 2 \} =m+n-3$$
and $\nu^{i_2}_{\infty,a_m+b_n-[a_m+b_n],m+n-4}(V\ten L)\geq 1$ we get $i_2=m+n-3$.
It follows now iteratively by repeating this argument that 
 the only possibility that the above derived Hodge numbers 
           match this Jordan decomposition is given as follows:
 $$\psi_\infty(V \ten L)\simeq  J^{m-1}(a_m,m)\ten J^{n-1} (b_n,n)=$$
$$      =J^{m+n-2}(a_m+b_n-[a_m+b_n],m+n-1) \oplus \dots \oplus J^{m+n-1-\min\{m,n\}}(a_m+b_n-[a_m+b_n],m+n+1-2\min\{m,n\}).$$ 
 This proves (i).
 
 Let $\psi_\infty(V)\simeq \bigoplus_{i,l} J^i(0,l)^{\nu^i_{\infty,0,l-1}(V)} \oplus \bigoplus_{(i,a,l): a\in(0,1)} J^i(a,l)^{\nu^i_{\infty,a,l-1}(V)}$.
 
 Then 
  $$\kappa^p_\infty(V \ten L(1,0))=\kappa^p_\infty(V)=\nu^p_{\infty,0,{ \rm prim}}(V)=\sum_l \nu^p_{\infty,0,l}(V)$$
  and by (i)
   $$\kappa^{p+1}_\infty(V\ten L(2,0))=\nu^{p+1}_{\infty,0,\prim}(V\ten L(2,0))=
   \sum_{l>0} \nu^{p+1}_{\infty,0,l}(V)+ \sum_{l} \nu^{p}_{\infty,0,l}(V) $$
 and 
  $$\kappa^{p+2}_\infty(V\ten L(0,3))=\nu^{p+2}_{\infty,0,\prim}(V\ten L(0,3))=
   \sum_{l>1} \nu^{p+2}_{\infty,0,l}(V)+ \sum_{l>0} \nu^{p+1}_{\infty,0,l}(V) +\sum_{l} \nu^{p}_{\infty,0,l}(V).$$
   
 Iterating this argument, one obtains 
   \begin{eqnarray*}\kappa^{p+r}_\infty(V\ten L(r+1,0))&=&\nu^{p+r}_{\infty,0,\prim}(V\ten L(r+1,0))\\
   &=&
   \sum_{l>r-1} \nu^{p+r}_{\infty,0,l}(V)+ \sum_{l>r-2} \nu^{p+r-1}_{\infty,0,l}(V)+\dots + \sum_{l} \nu^{p}_{\infty,0,l}(V).\end{eqnarray*}
   Hence one can recursively determine $\nu^{p+r}_{\infty,0,l}(V)$ starting with the $\nu^{j}_{\infty,0,\rk(V)-1}(V)$ for all $j$. 
   Analogously we proceed in case where $a \in (0,1)$.
 \end{proof}

 \begin{lemma}\label{FaltHyp} 
 Let $m,n\in \NN_{>0}$ and let $a_m,b_n\in \RR\cap (0,1).$ 
  Let $M_m, {N}_n$ be  irreducible hypergeometric Hodge modules of rank $m$, resp. $n$, 
  such that $\co_0(M_m)=m-1$ and  $\co_0(N_n)=n-1$ and  the local monodromy 
  at $\infty$ is a maximal Jordan block of the form 
  $\psi_\infty(M_m)\simeq J^{m-1}(a_m,m)$ and   $\psi_\infty(N_n)\simeq J^{n-1}(b_n,n)$.  (cf.~Def.~\ref{def1}).

   Then $M_m$ and $N_n$ are parabolically rigid, i.e.
   $$ H^1_{\rm par} (M_m)=0, \quad  H^1_{\rm par}(N_n)=0,$$
   and
    \[ \psi_\infty (M_m\tstar N_n)=\left\{ \begin{array}{cccc}
      \psi_\infty(M_m) \ten \psi_\infty(N_n) &\simeq&J^{m-1}(a_m,m)\ten J^{n-1} (b_n,n),& 0< a_m+b_n<1 \\
       (\psi_\infty(M_m) \ten \psi_\infty(N_n))(-1) &\simeq &(J^{m-1}(a_m,m)\ten J^{n-1} (b_n,n))(-1),& 1< a_m+b_n<2 \\
   \end{array}\right.,
\]
where
 $$  J^{m-1}(a_m,m) \ten J^{n-1}(b_n,n)=$$
$$      J^{m-1+n-1}(a_m+b_n-\lfloor a_m+b_n \rfloor,m+n-1) \oplus \cdots \oplus J^{m+n-1-\min\{m,n\}}(a_m+b_n-\lfloor a_m+b_n \rfloor,m+n+1-2\min\{m,n\}).$$

\end{lemma}

 \begin{proof}
 A hypergeometric Hodge module $H$ has singularities at $0, 1$ and $\infty$ (up to a Moebius transformation),
 where  the local monodromy at $1$ is a pseudo reflection,
 i.e. $\co_1(H)=1$, cf. Section 2, \cite{BH}.
 If $\co_0(H)=\rk(H)-1$
 we get 
 \[ \rk(H^1_{\rm par} (H))=\co_0(H)+\co_1(H)+\co_\infty(H)-2\rk(H)=0,\]
 which implies the first claim.
 
Assume first  $0< a_m+b_n<1$. Then $M_m \star N_n=M_m \tstar N_n$ by Corollary~\ref{KunF}~(i).
 In   the proof of Theorem~\ref{cohoFalt} 
 it was shown that  
  \begin{eqnarray*}
    h^{k}(M_m\tstar N_n)
                  &=&    \sum_{i}^{} (h^{i}(H^1_{\rm par} (M_m))+\kappa^{i-1}_\infty(M_m)) h^{k-i}(N_n) \\
                     &&   + \sum_{i}^{} h^i(M_m)(\delta^{k-1-i}(N_n)-\delta^{k-i}(N_n)+ \co^{k-1-i}_{\neq \infty}(N_n)) \\
                       && +o^{k-1}_\infty(M_m\ten N_n (t-x))-o^{k}_\infty(M_m\ten N_n (t-x))-\kappa^{k-1}_\infty(M_m\ten N_n (t-x)) \\
                       &=&  0  + \sum_{i}^{} h^i(M_m)(h^{k-i}(H^1_{\rm par}(N_n))+h^{k-i}(N_n)) +0-0-0\\
                       &=&  \sum_{i}^{} h^i(M_m)h^{k-i}(N_n) \\
                       &=&h^{k}(M_m \ten N_n(t-x)).
                          \end{eqnarray*}
                                  The  stationary phase formula (cf.~\cite{SabbahStationary}), Theorem~5.1,
                                   implies  that the Jordan blocks of $M_m\star N_n$ at infinity
           are
            $J(a_m+b_n,m+n-1),J(a_m+b_n,m+n-3),\ldots,J(a_m+b_n,m+n+1-2\min\{m,n\})$ which are exactly the Jordan blocks of the tensor product
            $J(a_m,m)\ten J(b_n,n)$ by Remark~\ref{Jordandec}.           
           The only possibility that the above derived Hodge numbers 
           match this Jordan decomposition is given as follows: 
      $$        \psi_\infty({M_m}\tstar {N_n}) = J^{m-1}(a_m,m)\ten J^{n-1} (b_n,n)$$
      $$=J^{m+n-2}(a_m+b_n,m+n-1) \oplus \dots \oplus J^{m+n-\min\{m,n\}}(a_m+b_n,m+n+1-2\min\{m,n\}).$$

  Assume now  $1< a_m+b_n<2$. As in   the proof of Theorem~\ref{cohoFalt} 
and using $o^l_\infty(M_m \ten N_n(t-x))=h^l(M_m \ten N_n(t-x))$ (cf.~Theorem~\ref{H-tensor}) one finds 
  \begin{eqnarray*}
    h^{k}(M_m\tstar N_n)
                  &=&    \sum_{i}^{} (h^{i}(H^1_{\rm par}(M_m))+\kappa^{i-1}_\infty(M_m)) h^{k-i}(N_n) \\
                     &&   + \sum_{i}^{} h^i(M_m)(\delta^{k-1-i}(N_n)-\delta^{k-i}(N_n)+ \co^{k-1-i}_{\neq \infty}(N_n)) \\
                       && +o^{k-1}_\infty(M_m\ten N_n (t-x))-o^{k}_\infty(M_m\ten N_n (t-x))-\kappa^{k-1}_\infty(M_m\ten N_n (t-x)) \\
                       &=&  0  + \sum_{i}^{} h^i(M_m)(h^{k-i}(H^1_{\rm par}(N_n))+h^{k-i}(N_n)) \\
                       && +    h^{k-1}(M_m \ten N_n(t-x))-h^{k}(M_m \ten N_n(t-x))
                       -0\\
                       &=&  \sum_{i}^{} h^i(M_m)h^{k-i}(N_n) +h^{k-1}(M_m \ten N_n(t-x))-h^{k}(M_m \ten N_n(t-x))\\
                       &=&h^{k-1}(M_m \ten N_n(t-x)).
                          \end{eqnarray*}
                          Using the stationary phase  as before the claim follows as in the case $0< a_m+b_n<1$.
  
\end{proof}

\begin{cor}\label{FaltHyp2}
 Let $m\in \NN_{>0}$ and let $a_m\in \RR \cap (0,1).$ 
 Let further $M_m$ be a parabolically rigid hypergeometric Hodge module of rank $m$
  with one non unipotent Jordan block $J^{m-1}(a_m,m)$ at $\infty$ of size $m$ and 
   $L$ be a Hodge module underlying a parabolically rigid  local system 
        without  unipotent Jordan blocks at $\infty$.
 If $L\ten M_m(t-x)$ has no unipotent Jordan block at infinity then  
 \begin{eqnarray*}
     \psi_\infty(L\tstar M_m) &\simeq&\bigoplus_{(k,a,l)} \left(J^k_{}(a,l)^{\nu^k_{\infty,a,l-1}(L)} \ten J^{m-1}_{}(a_m,m)\right)(-\lfloor a+a_m \rfloor).
    \end{eqnarray*}
\end{cor}

  
\begin{proof}
The claim is settled if $M_m=M_1$ a Kummer sheaf and therefore hypergeometric or if $L$ is hypergeometric by the previous result. Let now $L$ be non-hypergeometric and $m>1.$ Let $N_n$ be as in Lemma~\ref{FaltHyp} such that
  $a_m+b_n \not \in \ZZ $.
Then $M_m\star N_n=M_m\tstar N_n$ by Corollary~\ref{KunF}.
%
%
If  $N_n=N_1$ is a Kummer sheaf then
\begin{eqnarray*}
  (L\star M_m)\star N_1 &=& L\star (M_m\star N_1)\\ 
                 &=& L\star (M_m\tstar N_1)\\ 
                  &=& L\tstar (M_m\tstar N_1),\\ 
   \end{eqnarray*}
where the second equality uses that $M_m$ is not a Kummer sheaf and the third equality uses
that $M_m\star N_1=M_m\tstar N_1$ is a parabolically rigid irreducible hypergeometric Hodge module and $L$ is not hypergeometric.

On the right hand side of the last equation there appears no  skyscraper sheaf as a direct summand.
Hence  $(L\star M_m)\star N_1=(L\star M_m)\tstar N_1=(L\tstar M_m)\tstar N_1$, where
$L\star M_m =L\tstar M_m$ by Corollary~\ref{KunF}.
If $n>1$ we choose a $N_n$ with a residue $\mu$ at $0$ such that
$-\mu$ is not a residue of $L\tstar M_m$.
Hence by Theorem~\ref{delta}
$$(L\tstar M_m)\star N_n=(L\tstar M_m)\tstar N_n,$$
also for $n>1.$

Let  $H_{m+n+1-2k}$ be hypergeometric with 
 $$\psi_\infty(H_{m+n+1-2k})=J^{m+n-2k}(a_m+b_n-\lfloor a_m+b_n \rfloor,m+n+1-2k).$$
 By Corollary~\ref{KunF} and Lemma~\ref{FaltHyp}
if $0<a_m+b_n<1,$ 
             \begin{eqnarray*}
                   \kappa^l_\infty((L\tstar M_m) \ten N_n(t-x)) &=&  \kappa^l_\infty(L \ten (M_m \star N_n)(t-x)) \\
                  &=&  \sum_{k=1}^{\min\{m,n\}} \kappa^l_\infty((L \ten H_{m+n+1-2k} (t-x)) (-k+1)) \\
                  &=&  \kappa^l_\infty(L \ten (M_m \ten  N_n)) \\
                  &=&  \kappa^l_\infty((L \ten M_m) \ten  N_n) \\
  \end{eqnarray*}
and if $1<a_m+b_n<2,$
 \begin{eqnarray*}
                   \kappa^l_\infty((L\tstar M_m) \ten N_n(t-x)) &=&  \kappa^l_\infty(L \ten (M_m \star N_n)(t-x)) \\
                  &=&  \sum_{k=1}^{\min\{m,n\}} \kappa^l_\infty((L \ten H_{m+n+1-2k} (t-x)) (-k)) \\
                  &=&  \kappa^{l}_\infty(L \ten (M_m \ten  N_n)(-1)) \\
                  &=&  \kappa^{l}_\infty((L \ten M_m)(-1) \ten  N_n). \\
  \end{eqnarray*} 
 The claim follows now from  Remark~\ref{Jordandec}.

\end{proof}

 \begin{thm}  Let $V,L$ be the Hodge modules underlying 
  irreducible nonconstant  variations of complex polarized Hodge structures with 
  $$  \psi_{\infty}(V) \simeq \bigoplus_{(i,a,l)}   J^i_{}(a,l)^{ \nu^i_{\infty,a,l-1}(V)}\quad \textrm{ and }  
  \quad \psi_{\infty}(L) \simeq \bigoplus_{(j,b,m)}   J^j_{}(b,m)^{\nu^j_{\infty,b,m-1}(L)} .$$
  Then there is an isomorphism of 
  nilpotent orbits
   \begin{eqnarray*}
     \psi_\infty(V\tstar L) &\simeq &\bigoplus_{(i,j,a,b,l,m): a\neq 0,b \neq 0, a+b\neq 1} 
     J^i(a,l)^{ \nu^i_{\infty,a,l-1}(V)}  \ten J^j(b,m)^{\nu^j_{\infty,b,m-1}(L)}(-\lfloor a+b \rfloor)\\
     && \bigoplus_{(i,j,a,b,l,m): a\neq 0,b \neq 0, a+b=1}  \varphi\left(J^i(a,l)^{ \nu^i_{\infty,a,l-1}(V)} \ten J^j(b,m)^{\nu^j_{\infty,b,m-1}(L)}\right)(-1)\\
                         &&\bigoplus_{\quad \; (i,j,a,b,l,m):a=0,b \neq 0\; \quad} J^{i+1}(a,l+1)^{ \nu^i_{\infty,a,l-1}(V)} \ten J^j(b,m)^{\nu^j_{\infty,b,m-1}(L)}\\
                          &&\bigoplus_{\quad \;(i,j,a,b,l,m): a\neq 0,b =0\quad \;} J^i(a,l) ^{ \nu^i_{\infty,a,l-1}(V)}\ten J^{j+1}(b,m+1)^{\nu^j_{\infty,b,m-1}(L)}\\
                          &&\bigoplus_{\quad \;(i,j,a,b,l,m): a=0,b =0\quad \;} \varphi \left(J^{i+1}(a,l+1)^{ \nu^i_{\infty,a,l-1}(V)} \ten J^{j+1}(b,m+1)^{\nu^j_{\infty,b,m-1}(L)}\right)\\
                          &&\bigoplus_{\quad \quad\quad \quad (i,a,l)\quad \quad \quad \quad } J^{i}(a,l)^{ \nu^i_{\infty,a,l-1}(V)} \ten H^1_{{\rm par}}(L) \\
                           &&\bigoplus_{\quad \quad\quad \quad(j,b,m)\quad \quad\quad \quad} J^j(b,m)^{\nu^j_{\infty,b,m-1}(L)}\ten H^1_{\rm par}(V) \\
    \end{eqnarray*}
    where $\varphi(J^i(0,l)) :=J^{i-1}(0,l-1)$ where  $\varphi(J^i(a,l))= J^i(a,l)$ for $a\neq 0$ and the notion is 
    extended using direct sums, and where moreover 
    $$  J^i(a,l) \ten J^j(b,m)=
      J^{i+j}(a+b-\lfloor a+b \rfloor,l+m-1) \oplus \cdots \oplus J^{i+j+1-\min\{l,m\}}(a+b-\lfloor a+b \rfloor,l+m+1-2\min\{l,m\}).$$
 \end{thm}

 \begin{proof}
   Assume first that  $V,L$ are parabolically rigid without unipotent Jordan block at $\infty$ such that
   $V\ten L(t-x)$ has also no unipotent Jordan block at $\infty$.
   By Corollary~\ref{FaltHyp2}  there exists for each $n, a_n\in (0,1)$ a parabolic rigid irreducible hypergeometric $H_m(a_m)$ such that
   $\psi_\infty(H_n(a_n))\simeq J^{n-1}(a_n,n)$  and
   $$ (V\tstar L)\tstar H_n(a_n) = V\tstar (L\tstar H_n(a_n)).$$
   Hence, by Corollary~\ref{KunF}(ii)
   $$\kappa^l_\infty((V\tstar L)\ten H_n(a_n)(t-x))=\kappa^l_\infty(V\ten  (L\tstar H_n(a_n)(t-x))).$$
   Since these numbers determine uniquely the vanishing cycle structure of $V\tstar L$ at infinity by Remark~\ref{Jordandec} we obtain
   using Corollary~\ref{FaltHyp2}
     \begin{eqnarray}\label{eqtensorvl}
      \psi_\infty(V\tstar L) &\simeq&\bigoplus_{(i,j,a,b,l,m)}  J^i(a,l)^{ \nu^i_{\infty,a,l-1}(V)}  \ten J^j(b,m)^{\nu^j_{\infty,b,m-1}(L)}(-\lfloor a+b \rfloor),
    \end{eqnarray} as claimed.
    
       In the general situation we proceed  as follows:
   Let  $L_{\chi_1}, L_{\chi_2}$ be generic and $\mu_1,\mu_2 \sim 1$. 
   Then $V\star L_{\chi_1}$ and  $L\star L_{\chi_2}$ are parabolically rigid by Proposition~\ref{Lgeneric1} without unipotent Jordan block at $\infty$
   by Proposition~\ref{Lgeneric2} (i).
 Hence, by Corollary~\ref{Lgeneric2}(ii)
  \begin{eqnarray*}
   \psi_\infty(V\star L_{\chi_1})&\simeq &
   \bigoplus_{ (i,0,l)} J^{i+1} (1-\mu_1,l+1) ^{ \nu^i_{\infty,0,l-1}(V)}\bigoplus  _{ (i,a,l): a \neq 0 } J^{i} (a+1-\mu_1,l)^{ \nu^i_{\infty,a,l-1}(V)} \\ 
 &&   \bigoplus J^0(1-\mu_1,1)   \otimes H^1_{par}(V)
  \end{eqnarray*}
  and 
  \begin{eqnarray*}
   \psi_\infty(L\star L_{\chi_2})&\simeq&
   \bigoplus_{ (j,0,m)  } J^{j+1} (1-\mu_2,m+1) ^{\nu^j_{\infty,0,m-1}(L)}\bigoplus  _{ (j,b,m): b \neq 0} J^{j} (b+1-\mu_2,m)^{\nu^j_{\infty,b,m-1}(L)} \\ 
 &&   \bigoplus J^0(1-\mu_2,1)   \otimes H^1_{\rm par}(L).
  \end{eqnarray*}
By what was said above, the assumptions for Equation~\eqref{eqtensorvl} are now fulfilled with $V$ replaced by 
$V\star L_{\chi_1}$ and with $L$ replaced by $L\star L_{\chi_2}$ which proves the claim
of  the theorem for   $W:=(V\star L_{\chi_1}) \star (L\star L_{\chi_2}).$ 
Thus
\begin{eqnarray*}
\psi_\infty(W)&\simeq & \bigoplus_{ (i,j,0,0,l,m)} J^{i+1} (1-\mu_1,l+1)^{ \nu^i_{\infty,0,l-1}(V)}\  \ten J^{j+1} (1-\mu_2,m+1)^{\nu^j_{\infty,0,m-1}(L)} \\
                   && \bigoplus_{ (i,j,0,b,l,m):b \neq 0}  J^{i+1} (1-\mu_1,l+1)^{ \nu^i_{\infty,0,l-1}(V)}\  \ten J^{j} (b+1-\mu_2,m) ^{\nu^j_{\infty,b,m-1}(L)}\\
                   && \bigoplus_{ (i,0,l)}  J^{i+1} (1-\mu_1,l+1) ^{ \nu^i_{\infty,0,l-1}(V)} \ten J^0(1-\mu_2,1)   \otimes H^1_{{\rm par}}(L) \\
                  &&\\
                   &&   \bigoplus_{ (i,j,a,0,l,m):a\neq 0 } J^{i} (a+1-\mu_1,l)^{ \nu^i_{\infty,a,l-1}(V)}  \ten J^{j+1} (1-\mu_2,m+1)^{\nu^j_{\infty,0,m-1}(L)}  \\
                   && \bigoplus_{ (i,j,a,b,l,m):a\neq 0, b \neq 0}  J^{i} (a+1-\mu_1,l)^{ \nu^i_{\infty,a,l-1}(V)}  \ten J^{j} (b+1-\mu_2,m) ^{\nu^j_{\infty,b,m-1}(L)}\\
                   && \bigoplus_{ (i,a,l):a\neq 0}  J^{i} (a+1-\mu_1,l) ^{ \nu^i_{\infty,a,l-1}(V)}\ten J^0(1-\mu_2,1)   \otimes H^1_{\rm par}(L) \\
                  &&\\
                   && \bigoplus_{ (j,0,m)} J^0(1-\mu_1,1) \otimes H^1_{\rm par}(V) \ten J^{j+1} (1-\mu_2,m+1)^{\nu^j_{\infty,0,m-1}(L)} \\
                   && \bigoplus_{ (j,b,m):b \neq 0}  J^0(1-\mu_1,1) \otimes H^1_{\rm par }(V) \ten J^{j} (b+1-\mu_2,m)^{\nu^j_{\infty,b,m-1}(L)} \\
                   && \bigoplus_{ } J^0(1-\mu_1,1)\otimes H^1_{\rm par}(V)\ten J^0(1-\mu_2,1)   \otimes H^1_{\rm par}(L). \\
\end{eqnarray*}

    Hence, since $\mu_1, \mu_2 \sim 1$ are generic 
  the only residues (contained in $[0,1)$) that contribute to $\psi_\infty(W)$  are by Formula~\eqref{eqtensorvl}
  \[ a+b+1-\mu_1+1-\mu_2-[a+b], \quad 
   a+1-\mu_1+1-\mu_2,\quad  b+1-\mu_1+1-\mu_2,\quad  1-\mu_1+1-\mu_2\]
  with $a$, resp. $b$,  a non-zero residue of $V$, resp. $L$, at infinity.

  Using commutativity and  associativity of the middle convolution together with Theorem~\ref{delta}
  one finds    \[ W\star L_{\overline{\chi_1\chi_2}}
       =(V\star L)(-1) 
  .  \]
  By Proposition~\ref{Lgeneric2}(ii) we deduce that 
   a Jordan block $J^i(c,l)$ of $\psi_\infty(W)$  is transformed 
   to a Jordan block  of $\psi_\infty(W\star L_{\overline{\chi_1\chi_2}})$
   as follows:
%
%
%
  \[   J^i(c,l)\mapsto \left\{ \begin{array}{cc}
                                          J^{i}(0,l-1), & c=2-\mu_1-\mu_2    \\
                                           J^{i+1}(c+\mu_1+\mu_2-2,l), & c\neq 0, c\neq 2-\mu_1-\mu_2   \\
                                       \end{array} \right.\]
     which implies the expression for  $\psi_\infty(V\star L)= \psi_\infty(V\tstar L)$ in the theorem.   \end{proof}

\bibliographystyle{amsplain}
\bibliography{DR18}
\end{document}